\documentclass[11pt,a4paper,openright,oneside]{article}
\usepackage{amsfonts, amsmath, amssymb,latexsym,amsthm, mathrsfs, enumerate}
\usepackage[english]{babel}
\usepackage{epsfig}
\usepackage{xcolor,cite}
\usepackage{mathtools}
\usepackage{hyperref}
\usepackage[new]{old-arrows}
\usepackage{stackengine}
\usepackage{scalerel,wasysym}
\usepackage[normalem]{ulem}

\stackMath
\newcommand\reallywidetilde[1]{\ThisStyle{%
     \setbox0=\hbox{$\SavedStyle#1$}%
     \stackengine{-.1\LMpt}{\SavedStyle#1}{%
         \stretchto{\scaleto{\SavedStyle\mkern.2mu\AC}{.5150\wd0}}{.6\ht0}%
     }{O}{c}{F}{T}{S}%
}}

\newcommand{\Cay}{\mathrm{Cay}}
\newcommand{\End}{\mathrm{End}}
\newcommand{\Aut}{\mathrm{Aut}}

\usepackage[margin=1.2in]{geometry}
\usepackage{graphicx}
\usepackage{listings}
\usepackage[latin1]{inputenc}

\makeatletter
\providecommand\dotbigcup{\mathpalette\@barred\cdot}
\def\@barred#1#2{\ooalign{\hfil$#1\bigcup$\hfil\cr\hfil$#1#2$\hfil\cr}}

\makeatother

\numberwithin{equation}{section}
\newtheorem{teo}{Theorem}[section]
\newtheorem*{teo*}{Theorem}
\newtheorem*{prop*}{Proposition}
\newtheorem*{corol*}{Corollary}
\newtheorem{prop}[teo]{Proposition}
\newtheorem{corol}[teo]{Corollary}
\newtheorem{lema}[teo]{Lemma}
\newtheorem{lemma}[teo]{Lemma}
\newtheorem{defi}[teo]{Definition}

 \theoremstyle{definition}
 \newtheorem{case}{Case}

\newtheorem{remark}[teo]{Remark}
\newtheorem{example}[teo]{Example}

\usepackage{fancyhdr}

\begin{document}
\title{On endomorphism universality of sparse graph classes}
\author{Kolja Knauer\footnote{Departament de Matem\`atiques i Inform\`atica, Universitat de Barcelona, Spain} \footnote{LIS, Aix-Marseille Universit\'e, CNRS, and Universit\'e de Toulon, Marseille, France}\and Gil Puig i Surroca\footnote{Universit\'e Paris-Dauphine, Universit\'e PSL, CNRS, LAMSADE, 75016, Paris, France}}

\maketitle

\begin{abstract}
We show that every commutative idempotent monoid (a.k.a lattice) is the endomorphism monoid of a subcubic graph. This solves a problem of Babai and Pultr [J. Comb.~Theory, Ser.~B, 1980] and the degree bound is best-possible. On the other hand, we show that no class excluding a minor can have all commutative idempotent monoids among its endomorphism monoids. As a by-product we prove that monoids can be represented by graphs of bounded expansion (reproving a result of Ne\v{s}et\v{r}il and Ossona de Mendez) and $k$-cancellative monoids can be represented by graphs of bounded degree. Finally, we show that not all completely regular monoids can be represented by graphs excluding topological minor (strengthening a result of Babai and Pultr).
\end{abstract}

\begin{center} 
 
\end{center} 

\section{Introduction}

~\hfill All graphs, monoids, and posets are considered to be finite.

Already in 1936, K\H{o}nig raises the question, whether every group $\Gamma$ is isomorphic to the automorphism group $\Aut(G)$ of a graph $G$ \cite[Section 1.1]{Konig36}. Shortly after, Frucht gives a positive answer~\cite{Frucht1939}. Indeed, he later shows that every group is the automorphism group of a $3$-regular graph~\cite{Frucht49}. This is, the class of $3$-regular graphs is \emph{$\Aut$-universal}. Other $\Aut$-universal classes have been determined later, e.g., with given connectivity or chromatic number~\cite{S57} as well as chordal graphs~\cite{LB79} or comparability graphs of posets of bounded dimension~\cite{KZ15}. 
Several negative results into this direction come from concrete characterizations of groups appearing as automorphisms groups of a given graph class. The first such result is the characterization of automorphism groups of forests and can be traced back to Jordan, see P\'olya's paper~\cite{P37}. Similar results exist for bicyclic graphs~\cite{MA23}. Babai gave a classification of automorphism groups of planar graphs~\cite{B72I,B73II} which recently has been improved by Klav\'ik, Nedela, and Zeman~\cite{KLAVIK20221}. The same group of authors has announced characterizations of automorphism groups of outerplanar as well as graphs of treewidth $2$, see~\cite{KNZ15}. 
In particular, none of these classes is $\Aut$-universal. More generally, Babai~\cite{B74,B74b} shows that no graph class excluding a minor is $\Aut$-universal. A big step towards the understanding of automorphism groups of classes excluding a minor has been achieved recently by Grohe, Schweitzer, and Wiebking~\cite{grohe2020automorphism}. In particular, they show that any graph class excluding a minor represents only finitely many non-abelian simple groups, which resolves a conjecture of Babai~\cite{Babai1981}.

Already in the 60s, instead of the automorphism group $\Aut(G)$ of $G$, researchers become interested in analogous questions concerning the endomorphism monoid $\End(G)$ of $G$, and the representation of categories (see \cite{PT80} for a monograph on the subject). Hedrl\'in and Pultr show that every monoid can be represented as the endomorphism monoid of a graph~\cite{HP64,HP65}. A recent strengthening of this by the authors of the present paper is that every monoid can be represented as the endomorphism monoid of a regular graph~\cite{knauer2025rigidregulargraphsproblem}. 
Getting negative results by fully characterizing the endomorphism monoids of a given graph class is much harder than in the case of groups. Indeed, the structure of endomorphism monoids of paths~\cite{DFKQ20} and trees~\cite{CL14} are active areas of research. The main negative result in the area is due to Babai and Pultr~\cite{BP80} and says that no graph class excluding a topological minor is $\End$-universal. Note that this in particular includes graphs of bounded degree.

Because of the vastness of the class of monoids it is natural to distinguish important subclasses, see Figure~\ref{fig:venn} for some of the most commonly considered classes. Note that some of the regions displayed in the figure are empty.
We say that a class of graphs $\mathcal{G}$ is \emph{$\End$-universal for a class of monoids $\mathcal{M}$}, if every monoid $M\in\mathcal{M}$ is (isomorphic to) the endomorphism monoid $\End(G)$ for a graph $G\in\mathcal{G}$. 
A result of this type is by Hell and Ne\v{s}et\v{r}il~\cite{HN73}, who generalize Frucht's result~\cite{Frucht49} and show that $3$-regular graphs are $\End$-universal for the class of groups.
Another example into this direction is a paper of Koubek, R{\"o}dl and Shemmer~\cite{KRS09} studying small graphs representing monoids from a given class. Books that contain further information of representations of classes of monoids through graphs include~\cite{PT80,Kna-19,HN04}.

\begin{figure}[h]
\centering
\includegraphics[width=.5\textwidth]{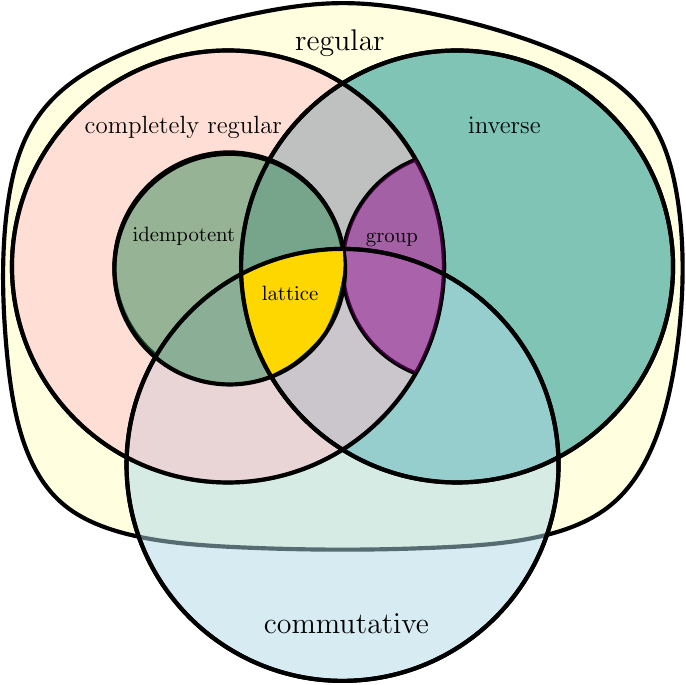}
\caption{Classes of monoids.}\label{fig:venn}
\end{figure}

In the present paper we study how sparse a class $\mathcal{G}$ can be in order to be $\End$-universal for a class of monoids $\mathcal{M}$. In order to measure sparsity we refer to a now well-known hierarchy of sparse graph classes, see the book of Ne\v{s}et\v{r}il and Ossona de Mendez~\cite{NOdM12}, that has proven to be central to Structural Graph Theory and Algorithms. Ne\v{s}et\v{r}il and Ossona de Mendez~\cite{NOdM17}  strengthened Hedrl\'in and Pultr's result by showing that there is a class $\mathcal G$ of bounded expansion which is $\End$-universal (for the class of all monoids). Indeed, most of the results we have mentioned so far with respect to $\Aut$ and $\End$ can be seen from this hierarchical point of view. We continue the study of $\End$-universality of sparse graph classes. See Figure~\ref{fig:sparseclasses} for an overview over some of the most important sparse classes and the results in this paper. 

\begin{figure}[h]
\centering
\includegraphics[width=.9\textwidth]{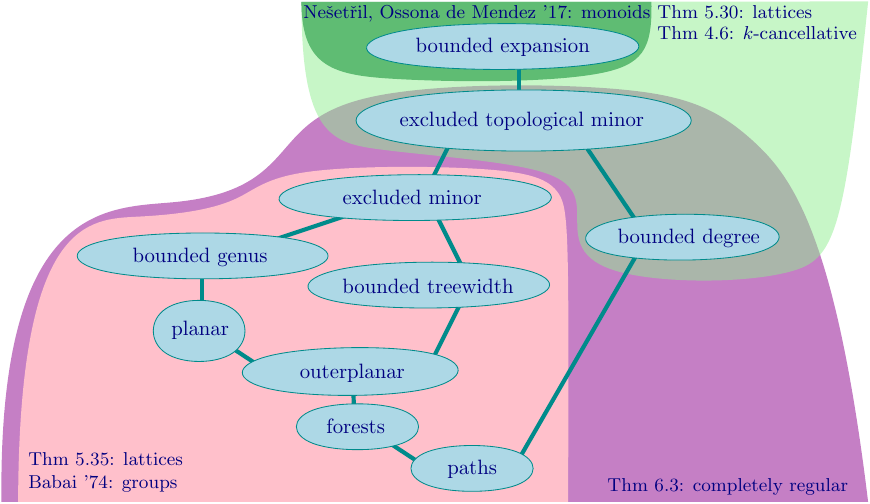}
\caption{Sparse classes of graphs. Green stands for endomorphism universality and violet for non-universality results.}\label{fig:sparseclasses}
\end{figure}

Answering a question of Babai and Pultr~\cite[Problem 1.4]{BP80} we show that graphs of maximum degree $3$ are $\End$-universal for commutative idempotent monoids (a.k.a lattices) (Theorem~\ref{teo:mainboundeddegree}), which is best-possible (Corollary~\ref{corol:bestpossible}). On the way we  complement Hell and Ne\v{s}et\v{r}il's result~\cite{HN73} by showing that every $k$-cancellative monoid is the endomorphism monoid of a simple undirected graph of maximum degree at most $\max(k+1,3)$ (Theorem~\ref{teo:groups}). 
In order to illustrate the techniques, we give a self-contained proof of the result of Ne\v{s}et\v{r}il and Ossona de Mendez~\cite{NOdM17}, presenting a class of (moreover linearly) bounded expansion that is $\End$-universal (Proposition~\ref{prop:boundedexpansion}).
On the other hand we strengthen the main result of Babai and Pultr~\cite{BP80}, by showing that no class excluding a topological minor is $\End$-universal for completely regular monoids (Theorem~\ref{teo:forcingatopologicalminor}).
Babai's result~\cite{B74,B74b} on $\Aut$-universality implies that no class excluding a minor is $\End$-universal for groups. Analogously to this, we show that no class excluding a minor is $\End$-universal for commutative idempotent monoids. Indeed, we show that already being $\End$-universal for distributive lattices does not allow excluding any minor (Theorem~\ref{teo:forcingaminor}). We finish the paper with an ample section of possible future directions. In particular, we discuss the variant of~\cite[Problem 1.4]{BP80} about non-commutative idempotent monoids. As a first step we show that bounded degree graphs are $\End$-universal for (adjoint monoids of) rectangular bands (Proposition~\ref{prop:band}).




%
%
%


%


\section{Preliminaries}

Let $K$ be a set of colors. A \emph{directed graph with $ K$-colored arcs} (or just arc-colored digraph) is a pair $D=(V,\{A_c\mid c\in K\})$, where $A_c\subseteq V^2$ for all $c\in K$. Thus, $D$ may be viewed as the arc-disjoint union of several digraphs---each without parallel arcs but possibly with loops---one for each color. More generally, for $K'\subseteq K$, we denote by $D[K']$ the maximal subdigraph of $D$ with $K'$-colored arcs, that is, $(V,\{A_c\mid c\in K'\})$.
We may think of a usual uncolored digraph $D=(V,A)$ as a digraph with $K$-colored arcs where $|K|=1$. Finally, an undirected graph $G=(V,E)$ without multiple edges may be thought of as a \emph{bidirected} digraph, where for each arc $(u,v)\in A$, also $(v,u)\in A$.  However, we usually denote edges of an undirected graph $G=(V,E)$ as sets $\{u,v\}\in E$. An undirected graph is \emph{simple} if it has no multiple edges and no loops. In the following we introduce some basic notions for arc-colored digraphs, which in case of uncolored or undirected graphs specialize to the standard notions unless mentioned otherwise.

A mapping $\varphi:V\to V$ is an \emph{endomorphism} of $D=(V,\{A_c\mid c\in K\})$ if $(u,v)\in A_c$ implies that $(\varphi(u),\varphi(v))\in A_c$ for all arcs of $D$. For any arc $(u,v)$ of $D$, we thus denote $(\varphi(u),\varphi(v))$ by $\varphi((u,v))$. A bijective endomorphism is an \emph{automorphism}. The set of endomorphisms of $D$ is denoted by $\End(D)$ and forms a monoid under composition. Similarly, the set of automorphisms $\Aut(D)$ of $D$ forms a group. Note that these definitions include the notions of endomorphisms and automorphisms of uncolored digraphs and undirected graphs.

An important source of arc-colored digraphs are Cayley graphs. Let $M$ be a monoid and $C\subseteq M$, then we define its arc-colored directed (right) \emph{Cayley graph} ${\Cay_{\mathrm{col}}}(M,C)$ with respect to $C$ as the digraph with $C$-colored arcs, where $V=M$ and $A_c=\{(m,mc)\mid m\in M\}$ for all $c\in C$. A well-known, see e.g.~\cite[Theorem 7.3.7]{Kna-19}, easy but important result, ties the monoids to their (generated) arc-colored directed Cayley graphs:

\begin{lemma}\label{lem:basic} Let $M$ be a monoid and let $C\subseteq M$ be a generating set of $M$. Then $M\cong\End({\Cay_{\mathrm{col}}}(M,C))$. 
\end{lemma}

Lemma~\ref{lem:basic} lies at the core of most constructions of a graph with prescribed endomorphism monoid or automorphism group. Basically, after applying the lemma, one only has to get rid of colors and orientations. Also  our positive results follow this idea. 
We continue with definitions of degrees in arc-colored digraphs.
For a digraph $D=(V,\{A_c\mid c\in K\})$ with $K$-colored arcs, $v\in V$ and $c\in K$, we denote by $\deg^+_c(v)$ and $\deg^-_c(v)$ the outdegree and indegree of $v$ with respect to the color $c$. We denote by $\Delta^+_c(D)$ and $\Delta^-_c(D)$ the maximum of such outdegrees/indegrees for $v\in V$. The \emph{outdegree}, \emph{indegree}, and \emph{degree} of $v$ are defined as $\deg^+(v)=\sum_{c\in K} \deg^+_c(v)$, $\deg^-(v)=\sum_{c\in K} \deg^-_c(v)$, and $\deg(v)=\deg^+(v)+\deg^-(v)$. 
The maximum and minimum of $\deg^+$, $\deg^-$ and $\deg$ over all vertices of $D$ are denoted by $\Delta^+(D)$, $\delta^+(D)$, $\Delta^-(D)$, $\delta^-(D)$, $\Delta(D)$ and $\delta(D)$.
Moreover, we denote by $\Delta^{\pm}(D)$ the quantity $\max\{\Delta_c^-(D),\Delta_c^+(D)\mid c\in K\}$ . Note that for degrees in undirected graph $G$ we will just use the standard definition, which differs from the one obtained by considering $G$ as a bidirected digraph.

Finally, we need to speak about walks in arc-colored digraphs.
Let $D=(V,\{A_c\mid c\in K\})$ be a digraph with $K$-colored arcs, and $ K'\subseteq K$. A \emph{$ K'$-walk} $W$ of $D$ is an alternating sequence $v_0,a_1,v_1,...,a_k,v_k$, where $v_0,...,v_k$ are vertices of $D$ and, for every $i\in\{1,...,k\}$, there is some $c_i\in K'$ with $a_i=(v_{i-1},v_i)\in A_{c_i}$. The \emph{set of vertices of} $W$ is $V(W)=\{v_0,...,v_k\}$, its \emph{set of arcs} is $A(W)=\{a_1,...,a_k\}$, and, for every $c\in K$, its \emph{set of $c$-arcs} is $A_c(W)=A(W)\cap A_c$. The nonnegative integer $k$ is the \emph{length} of $W$. 
We shall simply write \emph{walk} instead of \emph{$K'$-walk} if the context allows it, or if $ K'= K$.
We denote the \emph{concatenation} of walks $W$ and $W'$ such that the  last vertex of $W$ coincides with the first vertex of $W'$ by $WW'$. Note that, for $\varphi\in\End(D)$ 
and a directed $K'$-walk $W=(v_0,a_1,v_1,...,a_k,v_k)$ of $D$, the image $\varphi(W)$ of $W$ is a directed $K'$-walk $(\varphi(v_0),\varphi(a_1),\varphi(v_1),...,\varphi(a_k),\varphi(v_k))$. This will be an important tool to control the endomorphism monoids of arc-colored digraphs.



\section{Replacing colors by rigid gadgets and bounded expansion}\label{sec:HeldrlinPultr}
This section contains the proof of the following lemma, which in particular allows the construction of $\End$-universal graph classes of bounded expansion (Proposition~\ref{prop:boundedexpansion}):

\begin{lemma}\label{lem:third_step} For every arc-colored loopless digraph $D$ with $\delta^{+}(D), \delta^{-}(D)\geq 1$, there exists a simple undirected graph $G$ such that $\End(G)\cong\End(D)$ and $\Delta(G)=\max(\Delta(D),3)$.
\end{lemma}

In order to construct a digraph $D$ with given endomorphism monoid that satisfies the preconditions of Lemma~\ref{lem:third_step}, in the present section we will use Lemma~\ref{lem:basic2}, which is a slight strengthening of Lemma~\ref{lem:basic}. In the next section, however, since we have to further control the degree, we will use a blow-up construction, see Lemma~\ref{lem:second_step}.

%
%

For any $k>0$ the \emph{Hedrl\'in--Pultr graph} $H^k$ arises from the complete graph $K_4$ on vertices $b_0,b_1,b_2,b_3$ by subdividing once the edge $\{b_0,b_2\}$, twice the edge $\{b_0,b_3\}$,  $2k+1$ times the edge $\{b_1,b_2\}$, $2k-1$ times the edge $\{b_2,b_3\}$, and $2k$ times the edge $\{b_3,b_1\}$. See Figure~\ref{fig:Hg} for an illustration, where further we denote the bounded facial cycles of the plane drawing $Z_1, Z_2, Z_3$. The graph $H^1$ has been introduced by Hedrl\'in and Pultr~\cite{HP65}, the full family appears for instance in~\cite{Adams1981}, also see~\cite{N09}.

\begin{figure}[h]
\centering
\includegraphics[scale=1]{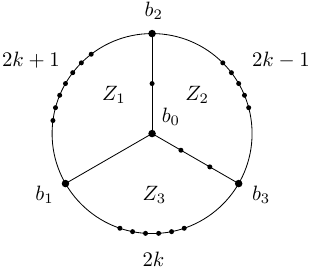}
\caption{The Hedrl\'in--Pultr graph $H^k$.}\label{fig:Hg}
\end{figure}

Note that the \emph{girth} $g(H^k)$, i.e.~the length of a shortest cycle of $H^k$, is $2k+5$, and $Z_1, Z_2, Z_3$ are the shortest cycles containing $b_1$ and $b_2$, $b_2$ and $b_3$, and $b_3$ and $b_1$, respectively. The importance of $H^k$ stems from the following property:

\begin{lemma}\label{lem:rigid} \emph{\cite{Adams1981}} For every $k>0$ the graph $H^k$ is \emph{rigid}, i.e., has no non-identity endomorphism.
\end{lemma}

Let $H^+$ be the set of vertices of $Z_1$ that are at distance at least $\frac{g(H^k)}{4}$ from $b_1$ and from $b_2$. Observe that $H^+$ forms a subpath of $Z_1$, and that $|H^+|\geq\frac{g(H^k)}{2}-4$. Similarly, let $H^-$ be the set of vertices of $Z_2$ at distance at least $\frac{g(H^k)}{4}$ from $b_2$ and $b_3$. If $g(H^k)\geq 9$ then $H^-$ also forms a subpath of $Z_2$, and $|H^-|\geq\frac{g(H^k)}{2}-6$. 

We will now use copies of $H^k$ as gadgets to replace colored arcs while preserving the endomorphism monoid. This kind of construction has been called the arrow-construction in~\cite[Section 2.3.2]{N99} and is also known as the \emph{\v{s}\'ip-product~\cite{M72}}. Variants have been used in~\cite{Frucht1939,HP64,HP65,HN73,HN05,VPH65}.

We will stick close to the standard notation.
So, given a set of colors $K$
, we choose $k$ large enough so that in $H^k$ the sets $H^+,H^-$ are of size at least $|K|$. For each $c\in K$, we pick vertices $c^+\in H^+$ and $c^-\in H^-$, such that each such vertex is picked for only one $c\in K$. We denote $H^k$ together with the choice of $c^+$ and $c^-$ by $H^k_c$. 

Let $D=(V,\{A_c\mid c\in K\})$ be a $K$-colored digraph, and $k$ large enough so that $H^k_c$ can be defined as above. We denote by $\check{D}_k$ the ($k$-th) \v{s}\'ip-product of $D$ and the Hedrl\'in--Pultr graph $H^k$. It is the simple undirected graph obtained from $D$ by replacing each arc $(x,y)\in A_c$ by a copy of $H^k_c$ and the vertices $x,y$ of $D$ are connected by an edge to $c^+,c^-$, respectively. See Figure~\ref{fig:new_gadget} for an illustration. 

\begin{figure}[h]
\centering
\includegraphics[scale=1]{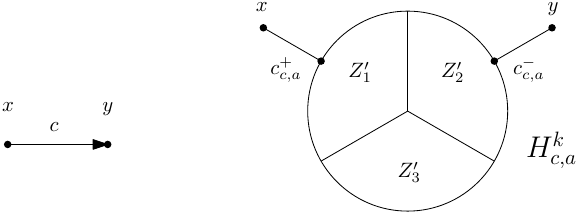}
\caption{A single-colored $c$-arc $a=(x,y)$ in $D$ and its replacement in $\check{D}_k$. The shortest cycles of $H^k_{c,a}$ are labelled $Z'_1,Z'_2,Z'_3$ in correspondence with those of $H^k$.}\label{fig:new_gadget}
\end{figure}

More formally, for each $c\in K$ and each $a\in A_c$ let $H^k_{c,a}=(V_{c,a},E_{c,a})$ be a copy of $H^k_c$ with designated vertices $c^+_{c,a},c^-_{c,a}$. If $X$ is such a copy and $x$  a vertex of $X$, we denote by $x_{c,a}$ the image of $x$ by the (by Lemma~\ref{lem:rigid}) unique homomorphism $X\rightarrow H^k_{c,a}$. Now, the ($k$-th) \v{s}\'ip-product of $D$ and $H^k$, or simply the \v{s}\'ip-product of $D$, is the undirected graph $\check{D}_k=(V',E')$ with:
\[V'=V\cup\bigcup_{\substack{c\in K \\ a\in A_c}} V_{c,a}, \]
\[E'=\{\{x,c^+_{c,a}\},\{y,c^-_{c,a}\}\mid c\in K,\ a\in A_c,\ a=(x,y)\}\cup\bigcup_{\substack{c\in K \\ a\in A_c}} E_{c,a},\]
where these unions are assumed to be disjoint.  
%
By construction we immediately get:

\begin{remark}\label{rem:degree} If $D$ is an arc-colored digraph with \v{s}\'ip-product $\check{D}_k$ then $\Delta(\check{D}_k)$ is at most $\max(\Delta(D),3)$. 
\end{remark}


Given $D=(V,\{A_c\mid c\in K\})$ with \v{s}\'ip-product $\check{D}_k=(V',E')$, for every $\varphi\in\End(D)$ we define $\varphi':V'\rightarrow V'$ as 
\[\varphi'(x)=\begin{cases}\varphi(x) &\text{if }x\in V \\ x_{c,\varphi(a)} &\text{if $x\in V_{c,a}$ for some $c\in K$ and $a\in A_c$.}\end{cases}\]

\begin{lemma}\label{lem:Phi3_well_defined} If $D$ is an arc-colored digraph with \v{s}\'ip-product $\check{D}_k$, and $\varphi\in\End(D)$, then $\varphi'\in\End(\check{D}_k)$.
\end{lemma}
\begin{proof} Let $D=(V,\{A_c\mid c\in K\})$, $\check{D}_k=(V',E')$ and $\{x,y\}\in E'$. We distinguish two cases.
\setcounter{case}{0}
\begin{case}
 One of $x$ and $y$ is in $V$ and the other in $V_{c,a}$ for some $c\in K$ and $a\in A_c$. 
\end{case}
Without loss of generality assume that $x\in V$ and $y=c^+_{c,a}$ for some $c\in K$, $a\in A_c$ and $x'\in V$ with $a=(x,x')$. We have that $\varphi'(y)=y_{c,\varphi(a)}=c^+_{c,\varphi(a)}$, and this already implies that $\{\varphi'(x),\varphi'(y)\}\in E'$.
If alternatively $y=c^-_{c,a}$ and  $a=(x',x)$, the argument is analogous.

\begin{case}
Both $x$ and $y$ are in $V_{c,a}$ for some $c\in K$ and $a\in A_c$.
\end{case}
We have $\{\varphi'(x),\varphi'(y)\}=\{x_{c,\varphi(a)},y_{c,\varphi(a)}\}\in E'$ by definition.
\end{proof}

\begin{lemma}\label{lem:Phi3_surjective} Let $D$ be an arc-colored loopless digraph with $\delta^{+}(D), \delta^{-}(D)\geq 1$ and $\check{D}_k$ its \v{s}\'ip-product. If $\varphi\in\End(\check{D}_k)$, then there is $\psi\in\End(D)$ such that $\varphi=\psi'$.
\end{lemma}
\begin{proof} Let $D=(V,\{A_c\mid c\in K\})$ and  $\check{D}_k=(V',E')$. Let $c\in K$ and $a\in A_c$. Since $D$ has no loops, by construction none of the vertices in $V$ is in a cycle of $\check{D}_k$ of length at most $g(H^k)$. This implies that $\varphi(V_{c,a})\cap V=\varnothing$. Since $H^k_{c,a}$ is connected, there exist $c'\in K$ and $a'\in A_{c'}$ such that $\varphi(V_{c,a})\subseteq V_{c',a'}$. Therefore the restriction of $\varphi$ to $V_{c,a}$ is the unique homomorphism $H^k_{c,a}\rightarrow H^k_{c',a'}$. 

Let $a=(x,y)$ and suppose for contradiction that $\varphi(x)\in V_{c',a'}$. Since $\delta^-(D)\geq 1$ we can take some $\tilde c\in K$ and $\tilde a\in A_{\tilde c}$ of the form $\tilde a=(x',x)$. The restriction of $\varphi$ to $V_{\tilde c,\tilde a}$ has to be the unique homomorphism $H_{\tilde c,\tilde a}\rightarrow H_{c',a'}$. In particular, $\varphi(c^+_{c,a})=c^+_{c',a'}$ and $\varphi(c^-_{\tilde c,\tilde a})=c^-_{c',a'}$ are two vertices of $H_{c',a'}$ at distance $2$ or less, a contradiction to the choice of $H^+$ and $H^-$. Therefore $\varphi(x)\in V$. Since $\delta^+(D)\geq 1$ any vertex of $V$ can take the role of $x$. Henceforth, $\varphi(V)\subseteq V$.

Since $\varphi(x)\in V$ is a neighbour of $\varphi(c^+_{c,a})=c^+_{c',a'}$, $c'=c$ and $\varphi(x)$ is the origin of $a'$. Similarly, since $\varphi(y)\in V$ is a neighbour of $\varphi(c^-_{c,a})=c^-_{c,a'}$, $\varphi(y)$ is the endpoint of $a'$.

Now let $\psi$ be the restriction of $\varphi$ to $V$. It is clear that $\psi\in\End(D)$, because $a'\in A_{c}$. It is also clear that $\psi'=\varphi$.
\end{proof}

\begin{prop}\label{prop:iso} Let $D$ be an arc-colored loopless digraph with $\delta^{+}(D), \delta^{-}(D)\geq 1$ and $\check{D}_k$ its \v{s}\'ip-product. The mapping $\Phi$ defined by $\Phi(\varphi)=\varphi'$ is a monoid isomorphism from $\End(D)$ to $\End(\check{D}_k)$.
\end{prop}
\begin{proof}  Let $D=(V,\{A_c\mid c\in K\})$ and  $\check{D}_k=(V',E')$. By Lemmas \ref{lem:Phi3_well_defined} and \ref{lem:Phi3_surjective} $\Phi(\End(D))=\End(\check{D}_k)$. It is clear that $\Phi$ is injective and that $\Phi(\textrm{id}_D)=\textrm{id}_{\check{D}_k}$. Now let $\varphi,\psi\in\End(D)$ and $x\in V'$. Then
\begin{align*}
(\Phi(\varphi)\circ\Phi(\psi))(x)&=\left.\begin{cases}\varphi'(\psi(x)) =\varphi(\psi(x))&\text{if }x\in V \\ \varphi'(x_{c,\psi(a)})=x_{c,\varphi(\psi(a))} &\text{if }x\in V_{c,a} \text{ for some } c\in K \text{ and } a\in A_c\end{cases}\right\}= \\
&=\Phi(\varphi\circ\psi)(x).
\end{align*}
\end{proof}

\begin{proof}[Proof of Lemma~\ref{lem:third_step}] As a consequence of Remark~\ref{rem:degree} and Proposition~\ref{prop:iso} we get that $\check{D}_k$ has maximum degree 
$\max(\Delta(D),3)$ and $\End(\check{D}_k)\cong \End(D)$.
\end{proof}


As a first application of the results of the present section we will construct a class of bounded expansion representing all monoids. Such a construction was already known to Ne{\v s}et{\v r}il and Ossona de Mendez~\cite{NOdM17}, but we present a self-contained proof here in order to illustrate these techniques. Moreover, our bound on the expansion seems more explicit. We need some definitions.

For the rest of this section denote by $|D|$ the number of vertices of an (arc-colored  directed) graph $D$ and by $\|D\|$ its number of edges/arcs (if arcs can have different colors, they are counted once for each color they have). The \emph{density} of an (arc-colored directed) graph $D$ is the quantity $\frac{\|D\|}{|D|}$. The \emph{radius} of an undirected, connected graph $G=(V,E)$ is defined as $\mathrm{rad}(G)=\underset{u\in V}{\min}\,\underset{v\in V}{\max}\, d(u,v)$, and a vertex $u$ with $\underset{v\in V}{\max}\,d(u,v)=\mathrm{rad}(G)$ is a \emph{central vertex}. For a non-negative integer $r$, an undirected graph $H$ is an \emph{$r$-shallow minor} of $G$ if there exists a partition $\{V_1,...,V_{\ell}\}$ of $V$ such that 
\begin{enumerate}[i)]
    \item $\mathrm{rad}(G[V_i])\leq r$ for each $1\leq i\leq \ell$;
    \item there is an injective mapping $v\mapsto V_{i(v)}$ from the vertex set of $H$ to $\{V_1,...,V_{\ell}\}$, satisfying that if two vertices $u, v$ of $H$ are adjacent then there is an edge between $V_{i(u)}$ and $V_{i(v)}$.
\end{enumerate}
A graph class $\mathcal G$ is said to have \emph{bounded expansion} if there is a function $f:\mathbb N\rightarrow\mathbb R $ such that any $r$-shallow minor of a graph in $\mathcal G$ has density at most $f(r)$. See~\cite{NOdM12} for equivalent definitions.



We give an easy improvement of Lemma~\ref{lem:basic} that will be useful for the proof of Proposition~\ref{prop:boundedexpansion}. The idea has been present in the literature, see e.g.~\cite[Theorem 7.4.4]{Kna-19} or~\cite[Figure 4.12]{HN04} and we indeed prove a stronger result later on, see Lemma~\ref{lem:second_step}, but we still propose this variant here, because of its simplicity. 
\begin{lemma}\label{lem:basic2} Let $M$ be a monoid of order $n$ and $C$ a generating set of $M$ with $|C|=m\leq n-1$. Then there exists a loopless arc-colored digraph $D$ with $\delta^{+}(D), \delta^{-}(D)\geq 1$ such that $M\cong\End(D)$. Moreover, $|D|=n(m+2)$, $\|D\|=2n(m+1)$, and the set of colors of the arcs has size $|K|=2(m+1)$.
\end{lemma}
\begin{proof}
We start with $C\subseteq M$ a generating set of $M$ and set $D''=\Cay_{\mathrm{col}}(M,C)$. Hence, by Lemma~\ref{lem:basic} we have $M\cong\End(D'')$. Now, we add a new color which consists of a loop at every vertex of $D''$ to obtain $D'$. Clearly, $\End(D')\cong\End(D'')$. Now, for each color $c$ of $D'$, let $c'$ be a new color, in a way that $c'_1\neq c'_2$ for different colors $c_1,c_2$. To obtain $D$ we subdivide each arc of $D'$ once in the following way: we replace it by a directed path of length $2$ (without the extreme vertices), the first arc of which (the one pointing towards the new vertex) has color $c$, and the other has color $c'$, where $c$ is the color of the original arc. Every endomorphism of $D'$ extends naturally to an endomorphism $\varphi$ of $D$ by mapping a subdivision vertex of a $c$-arc $(u,v)$ to the subdivision vertex of the $c$-arc $(\varphi(u),\varphi(v))$. Conversely, by the unique choice of colors $c,c'$ any endomorphism of $D$ must map subdivision vertices of $c$-arcs to subdivision vertices of $c$-arcs. This yields that all endomorphisms of $D$ arise from endomorphisms of $D'$ and allows to conclude $\End(D)\cong\End(D')$. Moreover, $D$ is loopless and $\delta^{+}(D), \delta^{-}(D)\geq 1$. The number of vertices, arcs and colors follows from the construction. 
\end{proof}

\begin{prop}\label{prop:boundedexpansion} Let $f:\mathbb R^+\rightarrow\mathbb R^+$ be a surjective, strictly increasing function. There exist an $\End$-universal graph class $\mathcal G$ and $r_0\in\mathbb R^+$ such that if $H$ is an $r$-shallow minor of a graph in $\mathcal G$ then $\frac{\|H\|}{|H|}\leq f(r)$ for any $r\geq r_0$.  
\end{prop}
\begin{proof}
Let $M$ be a monoid with $|M|=n$ having a generating set of size $m$. It can be assumed that $m+1\leq n$. By Lemma~\ref{lem:basic2} there is a loopless arc-colored digraph $D$ with $\delta^{+}(D), \delta^{-}(D)\geq 1$ and $M\cong\End(D)$. Moreover, $|D|=n(m+2)$ and $\|D\|=2n(m+1)$.



We now apply Proposition~\ref{prop:iso} to $D$ to construct the undirected graph $\check{D}_{k_n}$. The value of $k_n$ will be specified later. 
We have that
\begin{align*}
|\check{D}_{k_n}|&=|D|+|H^{k_n}|\|D\|=(12k_n+15)n(m+1)+n, \\
\|\check{D}_{k_n}\|&=(\|H^{k_n}\|+2)\|D\|=(12k_n+22)n(m+1).
\end{align*}

Let $H$ be an $r$-shallow minor of $\check{D}_{k_n}$ arising from a partition $\{V_1,...,V_{\ell},V_{\ell+1},...,V_{\ell+\ell'}\}$ of the vertex set of $\check{D}_{k_n}$, where $V_1,...,V_{\ell}$ are the sets corresponding to vertices of $H$. 



We will later choose $k_n$ as an increasing function of $n$ so our graphs satisfy the claim of the theorem. We first consider the case that $r\leq\frac{k_n}{6}<\frac{g(H^{k_n})}{12}$. Note that this implies $\mathrm{diam}(\check D_{k_n}[V_i])<\frac{g(H^{k_n})}{6}$ for any vertex set $V_i$ of the partition. Let $S_{\leq 6}$ be the set of vertices of $H$ with degree at most $6$, and let $S_{>6}$ be the set of the other vertices of $H$. Observe that, since $\mathrm{diam}(\check D_{k_n}[V_i])<\frac{g(H^{k_n})}{6}$, every element of $S_{>6}$ corresponds to a set $V_i$ that contains a vertex of $D$. 

\textbf{Claim.} $|S_{\leq 6}|\geq\sum_{v\in S_{>6}}\deg(v)$.

\noindent\textit{Proof of the Claim.}
Let $v,v'\in S_{>6}$ be two different vertices of $H$. We show $d_H(v,v')>2$, which implies the claim. Let $v=v_0,v_1,...,v_s=v'$ be the vertices of a shortest $H$-path between $v$ and $v'$, and let $V_{v_0},...,V_{v_s}$ be the sets of the partition from where they arise, respectively. Note that the pairs $V_{v_i},V_{v_{i+1}}$ are adjacent in $\check D_{k_n}$ for $i\in\{0,1,...,s-1\}$. Moreover, we know that there are vertices $x,x'$ of $D$ with $x\in V_v$, $x'\in V_{v'}$. Therefore, the distance between $x$ and $x'$ in $\check D_{k_n}$ is \[d_{\check D_{k_n}}(x,x')\leq\sum_{i=0}^{s}\mathrm{diam}
(\check D_{k_n}[V_{v_i}])+s<(s+1)\frac{g(H^{k_n})}{6}+s.\]
On the other hand, $d_{\check D_{k_n}}(x,x')\geq \frac{g(H^{k_n})}{2}+2$, as can be easily derived from the construction of the \v s\'ip-product $\check D_{k_n}$. This implies that $d_H(v,v')=s>2$.
\hfill $\blacksquare$

Therefore, if $r\leq\frac{k_n}{6}$ the  density of $H$ is 
\[\frac{\|H\|}{|H|}\leq\frac{6|S_{\leq 6}|+\sum_{v\in S_{>6}}\deg(v)}{2(|S_{\leq 6}|+|S_{>6}|)}\leq\frac{7|S_{\leq 6}|}{2(|S_{\leq 6}|+|S_{>6}|)}\leq \frac{7}{2}.\]

Let $f:\mathbb R^+\rightarrow\mathbb R^+$ be any given surjective, strictly increasing function. Call $h(x)$ the inverse of $\sqrt{\frac{f(x)}{7}}$. At last, define $k_n=6n\lceil h(n)\rceil$. Note that the condition $|K|\leq\frac{g(H^{k_n})}{2}-6$ for the construction of $\check{D}_{k_n}$, where $K$ is the set of colors of the arcs of $D$, is met ($|K|\leq 2n$ by Lemma~\ref{lem:basic2}). Now we set $\mathcal G$ as the class of graphs $\check{D}_{k_n}$ obtained in this way from all monoids. Let us check that for $r_0=h(1)$ the properties claimed in the theorem are satisfied. Consider any $r\geq r_0$, and let $n'$ be the integer with $h(n'-1)\leq r < h(n')$. If $|M|=n\geq n'$, then we have $ r\leq\frac{k_n}{6}$. Hence, the density of $H$ is bounded by $\frac{7}{2}\leq f(r_0)\leq f(r)$. If otherwise $|M|=n<n'$, we use that  $\check D_{k_n}$ is connected and bound the density $\frac{\|H\|}{|H|}$ of $H$ from above as follows:
\[\frac{\|\check{D}_{k_n}\|-\|\check{D}_{k_n}[V_1]\|-...-\|\check{D}_{k_n}[V_{\ell}]\|-(\|\check{D}_{k_n}[V_{\ell+1}]\|+1)-...-(\|\check{D}_{k_n}[V_{\ell+\ell'}]\|+1)}{\ell}\leq \]
\[\frac{\|\check{D}_{k_n}\|-|V_1|-...-|V_{\ell+\ell'}|+\ell}{\ell}\leq\frac{\|\check{D}_{k_n}\|-|\check{D}_{k_n}|+\ell}{\ell}\leq\frac{7n(m+1)-n+\ell}{\ell}\leq \]
\[7(n'-1)^2\leq 7h^{-1}(r)^2=f(r).\]

\end{proof}


\section{Blowing up vertices and representations of groups}\label{sec:blowup}
This section gives the proof of the following result, which as a by-product allows us to show that every $k$-cancellative monoid is the endomorphism monoid of a graph of maximum degree at most $\max(k+1,3)$ (Theorem~\ref{teo:groups}). It can be seen as an improvement of Lemma~\ref{lem:basic2}:
\begin{lemma}\label{lem:second_step} 
For every arc-colored digraph $D$ there exists an arc-colored loopless digraph $D'$ such that $\End(D')\cong\End(D)$ and 
\begin{enumerate}[i)]
   \item $\Delta(D')\leq\max(\Delta^{\pm}(D)+1,3)$;
   \item $\delta^+(D')=\delta^-(D')=1$.
\end{enumerate}
\end{lemma}


%


We will call the arc-colored digraph $D'$ the \emph{blow-up} of $D$. The idea for the construction of $D'$ is to blow up every vertex to a closed walk in order to locally reduce the degrees. But in order to preserve the endomorphism monoid new colors have to be introduced. See Figure~\ref{fig:blowup} for an illustration.

\begin{figure}[h]
    \centering
    \begin{minipage}{0.27\textwidth}
        \centering
        \includegraphics[width=1\textwidth]{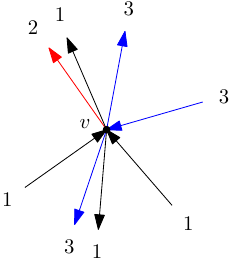}
    \end{minipage}\hfill
    \begin{minipage}{0.57\textwidth}
        \centering
        \includegraphics[width=1\textwidth]{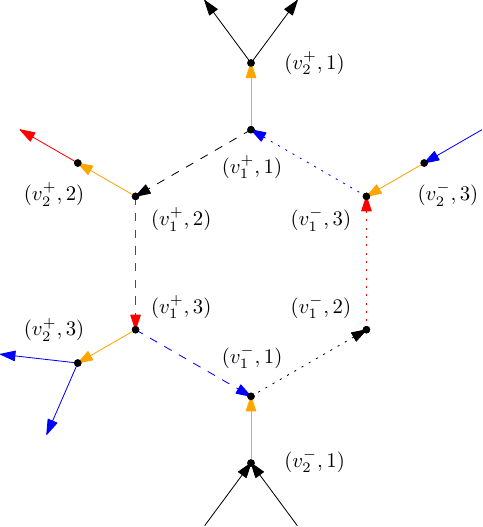}
    \end{minipage}
    \caption{Left: a vertex $v$ of $D$ with its outgoing and ingoing arcs and the labels of the colors. Right: the vertices of the walk $W_v$, their $0$-neighbours, and their ingoing and outgoing $K$-arcs in $D'$, assuming that $ K=\{1,2,3\}$. The arcs in $A'_0$ are orange, the arcs in $A'_{c^+}$ are dashed, the arcs in $A'_{c^-}$ are dotted, and the arcs in $A'_c$ are solid.}\label{fig:blowup}
\end{figure}

Let us now describe the construction more formally.
Let $D=(V,\{A_c\mid c\in K\})$. We define the set of colors of the arcs of $D'$ as $K'=\{c,c^+,c^-\mid c\in K\}\cup\{0\}$. 

Moreover, given $V$ we define new sets  $V_1^+=\{v_1^+\mid v\in V\}$, $V_1^-=\{v_1^-\mid v\in V\}$, $V_2^+=\{v_2^+\mid v\in V,\ \deg^+(v)\geq 1\}$ and $V_2^-=\{v_2^-\mid v\in V,\ \deg^-(v)\geq 1\}$. Finally, the set of vertices $V'$ of $D'$ is defined as $V'=(V_1^+\cup V_2^+\cup V_1^-\cup V_2^-)\times K$.

Assume that $K$ is the set of the first $|K|$ positive integers. To define the sets of arcs, for every $c\in K$ we set:
\[A'_0=\{((v_1^+,c'),(v_2^+,c'))\mid v_2^+\in V_2^+,\ c'\in K\}\cup\{((v_2^-,c'),(v_1^-,c'))\mid v_2^-\in V_2^-,\ c'\in K\},\]
\[A'_c=\{((u_2^+,c),(v_2^-,c))\mid (u,v)\in A_c\}, \]
\[A'_{c^+}=\begin{cases}
\{((v_1^+,c),(v_1^+,c+1))\mid v\in V\} &\text{if } c\neq | K| \\
\{((v_1^+,c),(v_1^-,1))\mid v\in V\} &\text{if } c=| K|,
\end{cases}\]
and, symmetrically,
\[A'_{c^-}=\begin{cases}
\{((v_1^-,c),(v_1^-,c+1))\mid v\in V\} &\text{if } c\neq | K| \\
\{((v_1^-,c),(v_1^+,1))\mid v\in V\} &\text{if } c=| K|.
\end{cases}\]

\bigskip

Let us first observe degrees and looplessness in $D'$:
\begin{lemma}\label{lem:degreesloops} Given an arc-colored digraph $D$, for the blow-up $D'$ the following hold:
\begin{enumerate}[i)]
   \item $D'$ is loopless, 
   \item $\Delta(D')\leq\max\,\{\Delta^{\pm}(D)+1,3\}$,
   \item $\delta^+(D')=\delta^-(D')=1$.
\end{enumerate}
\end{lemma}
\begin{proof} Let $D=(V,\{A_c\mid c\in K\})$. By construction of $D'=(V',\{A'_c\mid c\in K'\})$, every vertex $v\in V$ with a loop will be blown up to a walk $W_v$ of length at least $2$, and clearly no new loops are introduced, so $D'$ is loopless. To see (ii), let $v\in V$ and $c\in K$. From the construction of $D'$ it is clear that $\deg((v_1^+,c)),\deg((v_1^-,c))\leq 3$, that $\deg((v_2^+,c))=\deg^+_c(v)+1$ if $v_2^+\in V_2^+$ and that $\deg((v_2^-,c))=\deg^-_c(v)+1$ if $v_2^-\in V_2^-$. Part (iii) is also easy to see.
\end{proof}

Let us now describe the endomorphism monoid of $D'$. For any $\varphi\in\End(D)$, define $\varphi':V'\rightarrow V'$ as $\varphi'((v_i^*,c))=(\varphi(v)_i^*,c)$ for every $v\in V$, $c\in K$, $*\in\{+,-\}$ and $i\in\{1,2\}$. Note that since $\varphi\in\End(D)$ there are no problems with the definition when $i=2$.

\begin{lemma}\label{lem:Psi_well_defined}Let $D$ be an arc-colored digraph with blow-up $D'$. If $\varphi\in\End(D)$, then $\varphi'\in\End(D')$.
\end{lemma}
\begin{proof} Let $D=(V,\{A_c\mid c\in K\})$ with blow-up $D'=(V',\{A'_c\mid c\in K'\})$. It is clear that $\varphi'$ is compatible with all $( K'\setminus K)$-arcs. Now, for $c\in K$, consider a $c$-arc from a vertex $(u_2^+,c)$ to a vertex $(v_2^-,c)$, where $(u,v)\in A_c$. Since $\varphi\in\End(D)$, one has that $(\varphi(u),\varphi(v))\in A_c$, so indeed there is a $c$-arc from $(\varphi(u)_2^+,c)$ to $(\varphi(v)_2^-,c)$.
\end{proof}

Within the constructed arc-colored digraph $D'=(V',\{A'_c\mid c\in K'\})$ we will pay special attention to the \emph{blow-up} of a vertex $v$ for every $v\in V$, namely the closed $( K'\setminus K\setminus\{0\})$-walk $W_v$ defined by the sequence of vertices $(v_1^+,1),(v_1^+,2),...,(v_1^+,| K|)$, $(v_1^-,1),(v_1^-,2),...,(v_1^-,| K|),(v_1^+,1)$.

\begin{lemma}\label{lem:Psi_surjective} Let $D$ be an arc-colored digraph with blow-up $D'$. If $\varphi\in\End(D')$, then there exists $\psi\in\End(D)$ such that $\varphi=\psi'$.
\end{lemma}
\begin{proof} Let $D=(V,\{A_c\mid c\in K\})$ with blow-up $D'=(V',\{A'_c\mid c\in K'\})$. Let $u$ be any vertex of $D$. The image of $(u_1^+,1)$ has to be the starting point of a $1^+$-arc, so $\varphi((u_1^+,1))=(v_1^+,1)$ for some $v\in V$. This implies that $\varphi(W_u)=W_v$, which yields a mapping $\psi:V\rightarrow V$. Note that $\psi'=\varphi$. Now let $c\in K$ and consider any arc $(u,v)\in A_c$. By definition, $D'$ has the $c$-arc $((u_2^+,c),(v_2^-,c))$. Since $\varphi\in\End(D')$ and $\varphi=\psi'$, we can also say that $((\psi(u)_2^+,c),(\psi(v)_2^-,c))\in A'_c$. This implies that $(\psi(u),\psi(v))\in A_c$, so $\psi\in\End(D)$.
\end{proof}

\begin{prop}\label{prop:monoiso} Let $D$ be an arc-colored digraph with blow-up $D'$. The mapping $\Phi$ defined by $\Phi(\varphi)=\varphi'$ is a monoid isomorphism from $\End(D)$ to $\End(D')$.
\end{prop}
\begin{proof} Let $D=(V,\{A_c\mid c\in K\})$ with blow-up $D'=(V',\{A'_c\mid c\in K'\})$. By Lemmas \ref{lem:Psi_well_defined} and \ref{lem:Psi_surjective}, $\Phi(\End(D))=\End(D')$. It is clear that $\Phi$ is injective and $\Phi(\textrm{id}_{D})=\textrm{id}_{D'}$. Now let $\varphi,\psi\in\End(D)$. Then $\Phi(\varphi\circ\psi)$ is defined by 
\[(\varphi\circ \psi)'((v_i^*,c))=((\varphi\circ \psi)(v)_i^*,c)\]
for every $v\in V$, $c\in K$, $*\in\{+,-\}$ and $i\in\{1,2\}$, and $\Phi(\varphi)\circ\Phi(\psi)$ is defined by
\[(\varphi'\circ \psi')((v_i^*,c))=\varphi'(\psi(v)_i^*,c))=(\varphi(\psi(v))_i^*,c)\]
for every $v\in V$, $c\in K$, $*\in\{+,-\}$ and $i\in\{1,2\}$.
\end{proof}

Lemma~\ref{lem:degreesloops} together with Proposition~\ref{prop:monoiso} give that $D'$ has maximum degree at most $\max(\Delta^{\pm}(D)+1,3)$, minimum in- and outdegree $1$, and endomorphism monoid isomorphic to $\End(D)$; hence we have established Lemma~\ref{lem:second_step}.

Let us give an easy application of the so far obtained results. A monoid $M$ is \emph{right $k$-cancellative} if $|\{x\in M\mid xy=z\}|\leq k$ for all $y,z\in M$, \emph{left $k$-cancellative} if $|\{y\in M\mid xy=z\}|\leq k$ for all $x,z\in M$, and \emph{$k$-cancellative} if it is left $k$-cancellative or right $k$-cancellative.
We can use the so far obtained results to show:
\begin{teo}\label{teo:groups}
For every $k$-cancellative monoid $M$ there is a simple undirected graph $G$ of maximum degree at most $\max(k+1,3)$, such that $M\cong\End(G)$. 
\end{teo}
\begin{proof}
 Let $M$ be right $k$-cancellative. We take a generating set $C$ of $M$. By Lemma~\ref{lem:basic} we have $M\cong\End(\Cay_{\mathrm{col}}(M,C))$. Further, $D=\Cay_{\mathrm{col}}(M,C)$ has $\Delta^{\pm}(D)\leq k$. Thus, Lemma~\ref{lem:second_step} yields a loopfree blow-up say $D'$, which has maximum degree at most $\max(k+1,3)$, and minimum in- and outdegree $1$. Hence, we can apply Lemma~\ref{lem:third_step} to $D'$ and obtain a simple undirected graph $G$ of maximum degree at most $\max(k+1,3)$, such that $M\cong\End(G)$. If $M$ is left $k$-cancellative, we set $D$ to be the left-Cayley graph with respect to $C$ and apply the analogous version of Lemma~\ref{lem:basic} to follow the same proof.
\end{proof}

Since groups are $1$-cancellative Theorem~\ref{teo:groups} in particular yields that every group is the endomorphism monoid of a graph of maximum degree $3$. However, in~\cite{HN73} it is shown that every group is the endomorphism monoid of a $3$-regular graph, thus generalizing Frucht's result~\cite{Frucht49} and strengthening our result in this case. More generally, Babai and Pultr~\cite[Problem 2.3]{BP80} ask whether every monoid is the endomorphism monoid of a regular graph. {We gave a positive answer to that question, recently in~\cite{knauer2025rigidregulargraphsproblem}.}

\section{Lattices}

In this section we will focus on commutative idempotent monoids, i.e.~lattices. To prove the main results Theorems~\ref{teo:mainboundeddegree} and~\ref{teo:forcingaminor} we will first introduce some notions of posets and lattices that are common to both proofs. 

Let $(P,\leq)$ be a poset. 
If $x,y$ are incomparable we sometimes denote this as $x \parallel_P y$. 
For elements $x,y\in P$ we say that $y$ \emph{covers} $x$, if $x\neq y$, $x\leq y$,  and $x\leq z\leq y$ implies $z\in\{x,y\}$ for all $z\in P$. We denote this relation by $x\prec y$. The \emph{cover graph} $G_P$ of $P$ has vertex set $P$ and an edge $\{x,y\}$ if $x\prec y$. A common way to represent $G_P$, is by drawing it in the plane such that if $x\prec y$, then the edge goes upwards from $x$ to $y$. This drawing is usually called the \emph{Hasse diagram} of $P$, see Figures~\ref{fig:lattice} and~\ref{fig:thick} for examples.
A \emph{chain} of $P$ is a subset $Q\subseteq P$ which is totally ordered respect to $\leq$. An \emph{ideal} in a poset $P$ is a set $I\subseteq P$ with the property that $x\leq y\in I$ implies $x\in I$. An ideal $I$ is called \emph{principal} if there is a $y\in P$ such that $I=\{x\in P\mid x\leq y\}$. In this case we denote $I$ as $\downarrow_P y$. A  \emph{linear extension} of $\leq$ is a total order $\leq^*$ on the same set $P$ such that $x\leq y$ implies $x\leq^* y$.

A \emph{lattice} $L$ is a partially ordered set in which for every two elements $x,y\in L$ there is a unique largest element $z\leq x,y$, called \emph{meet} and denoted $x\wedge y$, and a unique smallest element $z\geq x,y$, called \emph{join} and denoted $x\vee y$. Every (finite) lattice has a unique maximum $\hat{1}$, and a unique minimum $\hat{0}$.  We consider a lattice $L$ as the monoid $(L,\wedge)$. This monoid is commutative and idempotent, and it is well-known that any commutative and idempotent monoid can be seen as lattice this way. We will adapt this view, because the order structure as well as the join-operation will turn out useful. Moreover, in Section~\ref{subsec:forcingaminor} we will recur to Birkhoff's Fundamental Theorem for Finite Distributive Lattices~\cite{B37}. Since join and meet are commutative and associative, for any $I=\{\ell_1, \ldots , \ell_k\}\subseteq L$, we will denote $\bigvee I:=\ell_1\vee \ldots \vee \ell_k$ and $\bigwedge I:=\ell_1\wedge \ldots \wedge \ell_k$. By convention $\bigvee\varnothing=\hat{0}$ and $\bigwedge\varnothing=\hat{1}$. 

We will now, introduce some definitions and a lemma, that are used formally only in Section~\ref{subsec:forcingaminor}, but serve to give an intuition for the constructions in Section~\ref{subsec:boundeddegree}. 
Given an (arc-coloured, directed) graph $D$, an endomorphism $\varphi\in\End(D)$ is a \emph{retraction} if its restriction to its image is the identity, i.e., $\varphi_{|\varphi(D)}=\mathrm{id}_{\varphi(D)}$. An (arc-coloured, directed) subgraph of $D$ that is induced by the image of a retraction is called a \emph{retract}. 
We collect the retracts of $D$ in the set $\mathcal{R}_D$. 

\begin{lema}\label{lem:End_is_lattice}
 Let $L$ be a lattice and $D$ an (arc-colored, directed) graph such that $\End(D)\cong L$. Then  $(L,\leq)\cong (\mathcal{R}_D,\subseteq)$. In particular, $\wedge$ corresponds to intersection of retracts. 
\end{lema}
\begin{proof}
 Since $L$ is idempotent, we have $\varphi(\varphi)=\varphi$ for all $\varphi\in \End(D)$. But this just means $\varphi_{|\varphi(D)}=\mathrm{id}_{\varphi(D)}$, i.e., all endomorphisms are retractions. Suppose now that there were two retractions $\varphi, \varphi'$ with the same image $R$. Since $L$ is commutative $\varphi'(v)=\varphi(\varphi'(v))=\varphi'(\varphi(v))=\varphi(v)$ for any vertex $v$. Thus, $\varphi=\varphi'$. We have established a bijection between $L$ and $\mathcal{R}_D$.
 
 Now $\varphi\leq_L \varphi'$ is equivalent to $\varphi\wedge \varphi'=\varphi$ which means $\varphi(\varphi')=\varphi$. It is straight-forward to check that this is equivalent to $\varphi(D)\subseteq \varphi'(D)$. We have shown $(L,\leq)\cong (\mathcal{R}_D,\subseteq)$.
\end{proof}






\subsection{Representing lattices with maximum degree three}\label{subsec:boundeddegree}
The main result of this section is that every lattice is the endomorphism monoid of a graph of maximum degree $3$ (Theorem~\ref{teo:mainboundeddegree}). Thus, we will provide the last missing ingredient to its proof:

\begin{lemma}\label{lem:first_step} 
For every lattice $L$ there exists an arc-colored digraph $D$ such that $\End(D)\cong L$ and $\Delta^{\pm}(D)\leq 2$.
\end{lemma}

Thus, throughout the section $L$ is a lattice with $|L|=n$ and the aim is to construct a digraph with $K$-colored arcs $D=(V,\{A_c\mid c\in K\})$ such that $\End(D)\cong L$ and $\Delta^{\pm}(D)\leq 2$.

\subsubsection{Construction of $D$}

We will regard $D$ as the superposition of seven arc-colored digraphs on the same set of vertices $V$. They will be constructed separately. The set of colors $K$ will be the disjoint union of their respective sets of colors, which we call $K_{\mathbf s}$, $K_{\mathbf r}$, $K_{\mathbf c}$, $K_{\mathbf{c'}}$, $K_{\mathbf{h}}$, $K_{\mathbf{i}}$ and $K_{\mathbf j}$. 
 %

\paragraph{The vertex set $V$:}

For the definition we fix an arbitrary linear extension
    $\leq^*$ of the order $\leq$ on $L$.
    Given any subset $S\subseteq L$ we denote by $S_i$ its $i$-th least element with respect to $\leq^*$ and by $S^i$ its $i$-th largest element with respect to $\leq^*$. In particular $L_1, \ldots, L_n$ is the set of all elements of $L$.

We extend the partial order $(L,\leq)$ to a set $L^+=L\cup\{0'\}$ 
, where $0'<x$ for every $x\in L\setminus\{\hat{0}\}$ and $0'\parallel \hat{0}$ are incomparable. Denote by $\mathcal Q_{L^+}$ the set of chains of $(L^+,\leq)$, and by $\mathcal Q'_{L^+}=\{Q\in\mathcal Q_{L^+}\mid Q\cap\{\hat{0},0'\}\neq\varnothing,\ Q\neq\{0'\}\}\subseteq\mathcal Q_{L^+}$ the set of chains containing a minimal element of $L^+$, except 
the single-element chain $\{0'\}$. For any $Q\in\mathcal Q'_{L^+}$ and any $x\in L$, define a new chain $[Q]_x\in\mathcal Q'_{L^+}$ as 
\[[Q]_x=\begin{cases}\{\hat{0}\} &\text{if } |Q|=2 \text{ and } ((Q_1=\hat{0} \text{ and } x<^*Q_2) \text{ or } ( Q_1=0' \text{ and } x\geq^*Q_2))\\ Q &\text{otherwise.}\end{cases}\]
Note that $[Q]_x=Q$ most of the time except for some chains of length $2$.  The vertex set of $D$ is defined as $V=\{([Q]_x,x)\mid Q\in\mathcal Q'_{L^+},\ x\in L\}\subseteq\mathcal Q'_{L^+}\times L$.

\paragraph{The $K_{\mathbf{s}}$-arcs:}
In the current subsection we define the $K_{\mathbf s}$-arcs of $D$. These are the arcs that sort of build the structure of $\End(D)$ from the point of view of Lemma~\ref{lem:End_is_lattice}. Indeed, after a sequence of definitions, we will get to the notion of \emph{petal} $W_{\{x\}}$---which is a walk in $D[K_{\mathbf s}]$ associated to $x\in L$, see Lemma~\ref{lem:petals}. The idea of the construction of $D[K_{\mathbf s}]$ is to guarantee that every $x\in L$ corresponds to a retraction mapping everything to $\bigcup_{y\in\downarrow_L x} W_{\{y\}}$, i.e., the union of petals of elements below $x$.  The colored arcs that will be introduced in the other subsections serve  to restrict the endomorphisms of $D$ such that apart from these no other endomorphisms are possible.

To every $x\in L$ we (bijectively) associate a color $\mathbf s(x)\in K_{\mathbf s}$.




\begin{defi}\label{def:Ks} For any $x\in L$, define
\[A_{\mathbf{s}(x)}=\begin{cases}\{(([Q]_x,x),([Q]_y,y))\mid Q\in\mathcal Q'_{L^+},\ x\prec^*y\} &\text{if } x\neq \hat{1} \\ \{(([Q]_{\hat{1}},\hat{1}),(\left[\tilde Q\right]_{\hat{0}},\hat{0}))\mid Q\in\mathcal Q'_{L^+}\} &\text{if } x=\hat{1},\end{cases}\]
where for any $Q\in\mathcal Q'_{L^+}$ the chain $\tilde Q\in\mathcal Q'_{L^+}$ is defined as follows:
\[\tilde Q=\begin{cases}
Q &\text{if } Q=\{\hat{0}\} \\
Q\cup\{(\downarrow_L Q_2)_2\} &\text{if } Q_1=\hat{0} \text{, } |Q|\geq 2 \text{ and } |\downarrow_L Q_2|\geq 3 \\
Q\cup\{0'\}\setminus\{\hat{0}\} &\text{if }  Q_1=\hat{0}\text{, } |Q|\geq 2 \text{ and } |\downarrow_L Q_2|=2 \\
\{\hat{0}\} &\text{if } Q_1=0'  \text{ and } |Q|=2 \\
Q\cup\{\hat{0},(\downarrow_L Q_3)_{i+1}\}\setminus\{0',Q_2\} &\text{if } Q_1=0' \text{, } |Q|\geq 3  \text{ and } |\downarrow_L Q_3|\geq i+2 \\
Q\setminus\{Q_2\} &\text{if } Q_1=0' \text{, } |Q|\geq 3 \text{ and } |\downarrow_L Q_3|=i+1,
\end{cases}\]
where $i$ is the positive integer such that $Q_2=(\downarrow_L Q_3)_i$, i.e., $Q_2$ is the $i$-th least element in $\downarrow_L Q_3$, and $(\downarrow_L Q_3)_{i+1}$ is the next element in $\downarrow_L Q_3$. 
\end{defi}

Since endomorphisms map walks to walks, in order to control the endomorphism monoid of $D$, we will specify a particular collection of walks. Let us do so, starting with a list of important properties derived from Definition~\ref{def:Ks}:
\begin{itemize}
    \item For every chain $Q\in\mathcal Q'_{L^+}$, we have $\left[\tilde Q\right]_{\hat{0}}=\tilde Q$.
    \item For every chain $Q\in\mathcal Q'_{L^+}$ the subgraph of $D[K_{\mathbf s}]$ induced by $\{([Q]_x,x)\mid x\in L\}\cup\{(\tilde Q,\hat{0})\}\subseteq V\subseteq \mathcal Q'_{L^+}\times L$ is a directed walk $W_Q$ of length $n$. Its vertices are ordered by increasing $x$ with respect to $\leq^*$, i.e., 
    $$([Q]_{\hat{0}},\hat{0})=([Q]_{L_1},L_1), ([Q]_{L_2},L_2), \ldots, ([Q]_{L_{n-1}},L_{n-1}),([Q]_{\hat{1}},\hat{1}), (\tilde Q,\hat{0}).$$ It has successive arcs of colors $\mathbf s(\hat{0})=\mathbf s(L_1),\mathbf s(L_2),...,\mathbf s(L_{n})=\mathbf s(\hat{1})$.
    \item The walks of the form $W_Q$ cover all vertices and arcs of $D[K_{\mathbf s}]$. 
    \item The last vertex $(\tilde Q,\hat{0})$ of $W_Q$ coincides with the first vertex of $W_{\tilde Q}$. In particular, $W_{\{\hat{0}\}}$ is closed (see Figure~\ref{fig:base_cycle3}) and
    $W_Q$ and $W_{\tilde Q}$ can be concatenated.
\end{itemize}

\begin{figure}[h]
\centering
\includegraphics[scale=1]{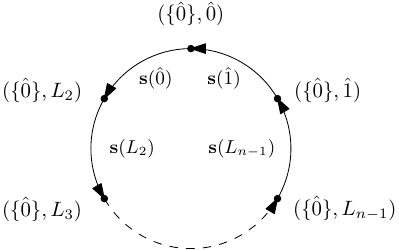}
\caption{$W_{\{\hat{0}\}}$.}\label{fig:base_cycle3}
\end{figure}


We will use the last property to extend the definition of these walks to general nonempty chains $Q\in L^+$ with $\hat{0},0'\notin Q$. 
The recursive nature of these walks can be observed in Figure~\ref{fig:x_cycle3}.

\begin{lemma}\label{lem:petals}
Let $Q$ be a non-empty chain of $L^+$ such that $\hat{0},0'\notin Q$. The concatenation of walks\[W_Q=W_{Q\cup\{(\downarrow_L Q_1)_1\}}W_{Q\cup\{(\downarrow_L Q_1)_2\}}...W_{Q\cup\{(\downarrow_L Q_1)_{|\downarrow_L Q_1|-1}\}}W_{Q\cup\{0'\}}\]  is well-defined, i.e., is a walk in $D[K_{\mathbf{s}}]$. If $Q=\{x\}$ for some $x\in L$, then $W_{\{x\}}$ is a closed walk called the \emph{petal} of $x$. 
\end{lemma}
\begin{proof}
In order to check that the concatenation is well-defined, we use that the chains $ Q\cup\{(\downarrow_L Q_1)_1\}=Q\cup\{\hat 0\}$ as well $ Q\cup\{0'\}$ from the concatenations are in $\mathcal Q'_{L^+}$. Thus, for any such chain $Q'$
we have that the last vertex $(\tilde Q',\hat{0})$ of $W_{Q'}$ coincides with the first vertex of $W_{\tilde Q'}$. More generally, for $1\leq i\leq |\downarrow_L Q_1|-1$ the last vertex of $W_{Q\cup\{(\downarrow_L Q_1)_i\}}$ is
\[\left.\begin{cases}
(\reallywidetilde{Q\cup\{\hat{0}\}},\hat{0}) & \text{if } i=1 \\
(\reallywidetilde{Q\cup\{0',(\downarrow_L Q_1)_i\}},\hat 0) & \text{if } 2\leq i\leq |\downarrow_L Q_1|-1
\end{cases}\right\}=\]
\[\begin{cases}
(Q\cup\{\hat{0},(\downarrow_L Q_1)_2\},\hat{0}) & \text{if } i=1 \text{ and } |\downarrow_L Q_1|\geq 3 \\
(Q\cup\{0'\},\hat{0}) & \text{if } i=1 \text{ and } |\downarrow_L Q_1|=2 \\
(Q\cup\{\hat{0},(\downarrow_L Q_1)_{i+1}\},\hat{0}) & \text{if } i\geq 2 \text{ and } |\downarrow_L Q_1|\geq i+2 \\
(Q\cup\{0'\},\hat{0}) & \text{if } i\geq 2 \text{ and } |\downarrow_L Q_1|=i+1, \\
\end{cases}\]
coinciding with the first vertex of $W_{Q\cup\{(\downarrow_L Q_1)_{i+1}\}}$ when $1\leq i\leq |\downarrow_L Q_1|-2$, and with the first vertex of $W_{Q\cup\{0'\}}$ when $i=|\downarrow_L Q_1|-1$. 

Finally, note that if $|Q|=1$, then $W_Q$ is a closed walk, since its starting vertex $(\{\hat{0}\},\hat{0})$ is also its last vertex. 
\end{proof}

\begin{figure}[h]
\centering
\includegraphics[width=\textwidth]{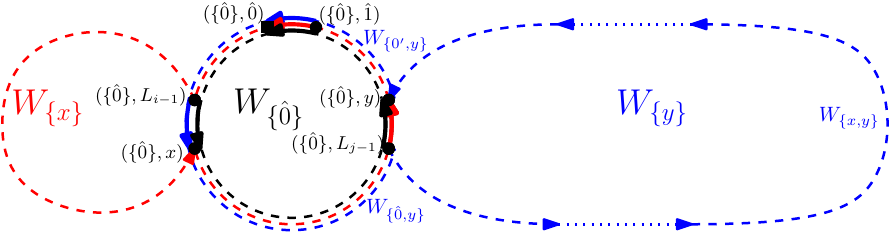}
\caption{Illustration of $W_{\{\hat{0}\}}$ (black), $W_{\{x\}}$ (red), $W_{\{y\}}$ (blue) for $\hat{0}<x<y$. The petal $W_{\{y\}}$ contains as constituents walks $W_{\{\hat{0},y\}}$ starting at $(\{\hat{0}\},\hat{0})$, $W_{\{0',y\}}$ ending at $(\{\hat{0}\},\hat{0})$, and  $W_{\{x,y\}}$, which is a copy of the petal $W_{\{x\}}$, except that it is not closed.  More concretely, the starting vertex $(\{\hat 0,x,y\},\hat 0)$ of $W_{\{x,y\}}$ and its ending vertex $(\widetilde{\{0',x,y\}},\hat 0)$ 
correspond both to $(\{\hat 0\},\hat 0)$, while for the rest of the vertices there is the bijective correspondence $(Q,z)\rightarrow([Q\setminus\{y\}]_z,z)$. Here $i,j$ are the integers such that $L_{i}=x$, $L_{j}=y$.}\label{fig:x_cycle3}
\end{figure}

As it can be seen in Figure~\ref{fig:x_cycle3}, the maximum indegree and outdegree with respect to any color of $K_{\mathbf s}$ is low.
\begin{remark}\label{rem:A_s-arcs} Let $x\in L$. Then $\Delta^+_{\mathbf s(x)}(D)=\Delta^-_{\mathbf s(x)}(D)\leq 2$.
\end{remark}
\begin{proof} Let $i$ be the index such that $x=L_{i}$, and let $x^+=\begin{cases}L_{i+1} &\text{if } x\neq {\hat{1}} \\ \hat 0 &\text{if } x={\hat{1}}. \end{cases}$ Let $u$ be a vertex of $D$. Write $u=([Q]_y,y)$, where $y\in L$ and $Q\in\mathcal Q'_{L^+}$. Note that a necessary condition for $u$ to have $\mathbf s(x)$-out-neighbours is that $x=y$. In that case, the $\mathbf s(x)$-out-neighbours of $u$ are of the form $([Q']_{y^+},y^+)$ if $y\neq {\hat{1}}$, where $Q'\in\mathcal Q'_{L^+}$ is such that $[Q]_y=[Q']_y$, and of the form $(\left[\tilde{ Q'}\right]_{y^+},y^+)$ if $y={\hat{1}}$. We have to be a bit careful when $[Q]_y=[Q']_y$ but $Q\neq Q'$, that is, when $[Q]_y=\{\hat 0\}$. In this situation, the $\mathbf s(x)$-out-neighbours of $u$ are $(\{\hat 0,y^+\},y^+)$ and $(\{\hat 0\},y^+)$ when $y\neq\hat 1$, and just $(\{\hat 0\},\hat 0)$ when $y=\hat 1$. Thus,

\[\deg^+_{\mathbf s(x)}(u)=\begin{cases}
2 &\text{if } [Q]_y=\{{\hat{0}}\} \text{, } y=x \text{ and } x\neq {\hat{1}} \\
1 &\text{if } [Q]_y\neq\{{\hat{0}}\} \text{ and } y=x \text{, or if } y=x={\hat{1}} \\
0 &\text{if } y\neq x.
\end{cases}\]
With a similar analysis of the $\mathbf s(x)$-in-neighbors of $u$, it is concluded that
\[\deg^-_{\mathbf s(x)}(u)=\begin{cases}
2 &\text{if } [Q]_y=\{{\hat{0}}\} \text{, } y=x^+ \text{ and } x\neq {\hat{1}} \\
1 &\text{if } [Q]_y\neq\{{\hat{0}}\} \text{ and } y=x^+ \text{, or if } y=x^+={\hat{0}} \\
0 &\text{if } y\neq x^+.
\end{cases}\]
\end{proof}

\begin{example}\label{example_s}
We illustrate the construction so far with a small example. Consider the lattice $(L,\leq)$  in Figure~\ref{fig:lattice}.

\begin{figure}[h]
    \centering
    \includegraphics[width=.2\textwidth]{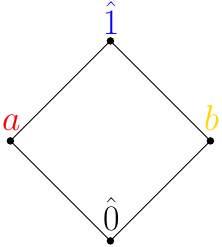}
    \caption{A lattice.}\label{fig:lattice}
\end{figure}

This yields that the set $\mathcal Q'_{L^+}$ consists of the chains: 
    $$\{{\hat{0}}\},\{{\hat{0}},a\},\{0',a\},\{{\hat{0}},b\},\{0',b\},\{{\hat{0}},{\hat{1}}\},\{0',{\hat{1}}\},\{{\hat{0}},a,{\hat{1}}\},\{0',a,{\hat{1}}\},\{{\hat{0}},b,{\hat{1}}\},\{0',b,{\hat{1}}\}.$$

\begin{figure}[h]
    \centering
    \includegraphics[width=1\textwidth]{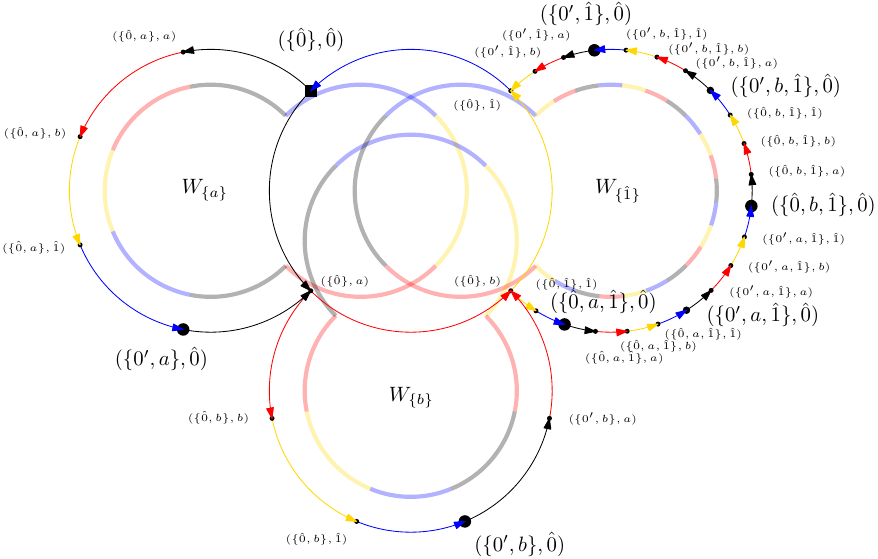}
    \caption{An example for the construction of  $D[K_{\mathbf{s}}]$ and the set of petals.}\label{fig:example_s}
\end{figure}

 Let furthermore,  ${\hat{0}}\leq^*a\leq^*b\leq^*{\hat{1}}$ be the designated linear extension of $\leq$. The resulting arc-colored digraph $D[K_{\mathbf s}]$ is depicted in Figure~\ref{fig:example_s}. The vertex $(\{{\hat{0}}\},{\hat{0}})$ is the starting and ending point of all the petals. The starting vertices of walks $W_{Q}$ are drawn larger than the other vertices, if $|Q|=2$ more than if $|Q|=3$. The $\mathbf s({\hat{0}})$-arcs are black, the $\mathbf s(a)$-arcs red, the $\mathbf s(b)$-arcs yellow, and the $\mathbf s({\hat{1}})$-arcs blue. Petals $W_{\{a\}}$, $W_{\{b\}}$, $W_{\{{\hat{1}}\}}$ are marked with background shadows. The walk $W_{\{{\hat{0}}\}}$ is just the central cycle.

\end{example}

\paragraph{The $ K_{\mathbf r}$-arcs:} Now we define the $K_{\mathbf r}$-arcs of $D$. To every $x\in L\setminus\{\hat 0\}$ we associate a color $\mathbf r(x)\in K_{\mathbf r}$ in an arbitrary bijective manner. Let  $\mathcal Q_{\mathbf r(x)}=\{Q\in\mathcal Q'_{L^+} \mid |Q|\geq 2,\ Q_2\neq x\}\cup\{\{{\hat{0}}\}\}$. The set of $\mathbf r(x)$-arcs is defined as 
\[A_{\mathbf r(x)}=
\{(([Q]_y,y),([Q]_y,y))\mid Q\in \mathcal Q_{\mathbf r(x)},\ y\in L\}.\]
Thus all arcs of $D[K_{\mathbf r}]$ are loops. Consequently,
\begin{remark}\label{rem:A_r-arcs} Let $x\in L\setminus\{\hat 0\}$. Then $\Delta^+_{\mathbf r(x)}(D)=\Delta^-_{\mathbf r(x)}(D)=1$.
\end{remark}


Later, it will be useful to know that certain vertices are $\mathbf r(x)$-loopless. We add another easy lemma.
\begin{lemma}\label{lem:anticolors} Let $x,y\in L$ and ${\hat{0}}<x$. Vertices of the form $(\{{\hat{0}},x\},y)$ or $(\{0',x\},y)$ have no $\mathbf r(x)$-loop.  
\end{lemma}
\begin{proof}
Note that both $\{{\hat{0}},x\}$ and $\{0',x\}$ are chains of length $2$ with top element $x$, hence they are not in $\mathcal Q_{\mathbf r(x)}$. Hence, they cannot lie in the image $[Q]_y$ for any $Q\in\mathcal Q_{\mathbf r(x)}$, since if $[Q]_y\neq Q$, then $[Q]_y=\{{\hat{0}}\}$. 
\end{proof}

\begin{figure}[h]
\centering
\begin{minipage}{0.15\textwidth}
        \centering
        \includegraphics[width=1\textwidth]{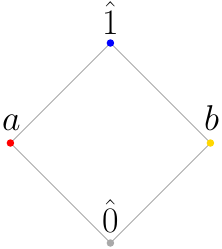}
    \end{minipage}\hfill
    \begin{minipage}{0.7\textwidth}
        \centering
        \includegraphics[width=1\textwidth]{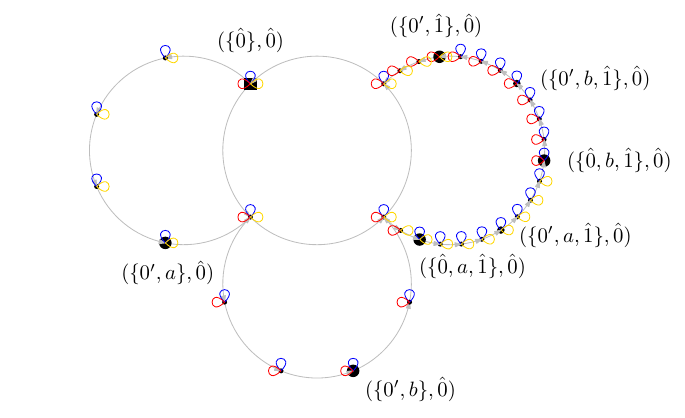}
    \end{minipage}
\caption{Illustration of $D[K_{\mathbf r}
]$ as a continuation of Example~\ref{example_s}. On the left, a drawing of $L$, where elements leading to $K_{\mathbf r}$-arcs are highlighted. On the right, $D[K_{\mathbf r}]$, with $\mathbf r(a)$-arcs in red, $\mathbf r(b)$-arcs in yellow and $\mathbf r({\hat{1}})$-arcs in blue. $K_{\mathbf s}$-arcs appear in grey as a reference. Note that in this particular example $\mathcal Q_{\mathbf r(a)}=\{\{{\hat{0}}\},\{{\hat{0}},b\},\{0',b\},\{{\hat{0}},{\hat{1}}\},\{0',{\hat{1}}\},\{{\hat{0}},b,{\hat{1}}\},\{0',b,{\hat{1}}\}\}$ and
$\mathcal Q_{\mathbf r({\hat{1}})}=\{\{{\hat{0}}\},\{{\hat{0}},a\},\{0',a\},\{{\hat{0}},b\},\{0',b\},\{{\hat{0}},a,{\hat{1}}\},\{0',a,{\hat{1}}\},\{{\hat{0}},b,{\hat{1}}\},\{0',b,{\hat{1}}\}\}$. }\label{fig:example_r}
\end{figure}

\paragraph{The $ K_{\mathbf c}$-arcs:} 
To every strictly comparable pair $(x,y)\in L\times L$ with $\hat{0}<x<y$ we associate a color $\mathbf c(x,y)\in K_{\mathbf c}$ in an arbitrary bijective manner. Let $\mathcal Q_{\mathbf c(x,y)}=\{Q\in\mathcal Q'_{L^+}\mid |Q|\geq 2,\ Q_1={\hat{0}},\ Q_2=x,\ Q^1\leq y\}\cup\{\{{\hat{0}}\}\}$. The set of $\mathbf c(x,y)$-arcs is defined as \[A_{\mathbf c(x,y)}=\{((Q,y),([Q]_{\hat{0}},{\hat{0}}))\mid Q\in\mathcal Q_{\mathbf c(x,y)}\}.\]

\begin{figure}[h]
\centering
\begin{minipage}{0.15\textwidth}
        \centering
        \includegraphics[width=1\textwidth]{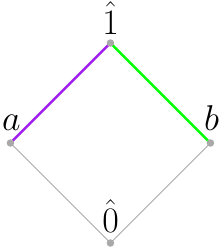}
    \end{minipage}\hfill
    \begin{minipage}{0.7\textwidth}
        \centering
        \includegraphics[width=1\textwidth]{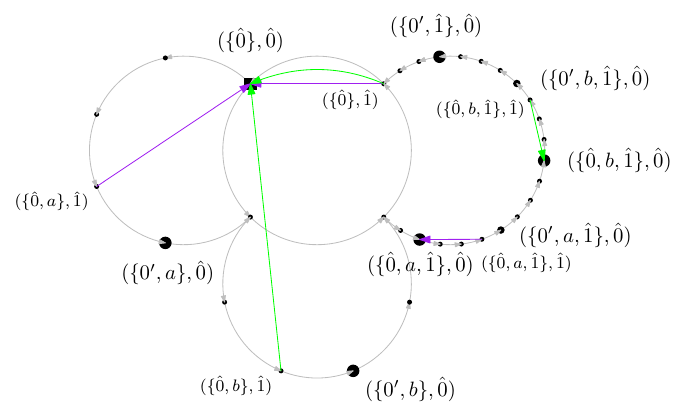}
    \end{minipage}
\caption{Illustration of 
$D[K_{\mathbf c}
]$  as a continuation of Example~\ref{example_s}. 
On the left, a drawing of $L$, where the comparable pairs leading to $K_{\mathbf c}$-arcs are highlighted. On the right, $\mathbf c(a,{\hat{1}})$-arcs in purple and $\mathbf c(b,{\hat{1}})$-arcs in green. Note that $\mathcal Q_{\mathbf c(a,{\hat{1}})}=\{\{{\hat{0}}\},\{{\hat{0}},a\},\{{\hat{0}},a,{\hat{1}}\}\}$.}\label{fig:example_c}
\end{figure}

\begin{remark}\label{rem:A_c-arcs} Let $x,y\in L$ with $\hat 0<x<y$. Then $\Delta^+_{\mathbf c(x,y)}(D)=1$ and $\Delta^-_{\mathbf c(x,y)}(D)\leq 2$.
\end{remark}
\begin{proof} 
It is clear that all $\mathbf c(x,y)$-arcs have different origins. Suppose that two different $\mathbf c(x,y)$-arcs, say $((Q,y),([Q]_{\hat{0}},{\hat{0}}))$ and $((Q',y),([Q']_{\hat{0}},{\hat{0}}))$, have the same endpoint. Here $Q,Q'\in\mathcal Q_{\mathbf c(x,y)}$, $[Q]_{\hat{0}}=[Q']_{\hat{0}}$ and $Q\neq Q'$. Without loss of generality, we can assume that $[Q]_{\hat{0}}\neq Q$. Then $|Q|=2$ and $[Q]_{\hat{0}}=\{{\hat{0}}\}$. Therefore $Q=\{{\hat{0}},x\}$ and $Q'=\{{\hat{0}}\}$.
\end{proof}

\paragraph{The $K_{\mathbf {c'}}$-arcs:} The set of $K_{\mathbf {c'}}$-arcs of $D$ is similar to the set of $K_{\mathbf{c}}$-arcs and has a similar purpose. To every pair $(x,y)\in L\times L$ with $\hat{0}<x<y$ we associate a color $\mathbf{c'}(x,y)\in K_{\mathbf{c'}}$ in an arbitrary bijective manner. Let $\mathcal Q_{\mathbf{c'}(x,y)}=\{Q\in\mathcal Q'_{L^+}\mid |Q|\geq 2,\ Q_1=0',\ Q_2=x,\ Q^1\leq y\}\cup\{\{{\hat{0}}\}\}$. The set of $\mathbf{c'}(x,y)$-arcs is defined as \[A_{\mathbf {c'}(x,y)}=\{((Q\setminus\{0'\}\cup\{\hat 0\},y),(Q,{\hat{0}})), ((Q\setminus\{0'\}\cup\{\hat 0\},y),([Q]_{\hat{1}},\hat{1}))\mid Q\in\mathcal Q_{\mathbf {c'}(x,y)}\}.\]

\begin{figure}[h]
\centering
\begin{minipage}{0.15\textwidth}
        \centering
        \includegraphics[width=1\textwidth]{example_lattice2c}
    \end{minipage}\hfill
    \begin{minipage}{0.7\textwidth}
        \centering
        \includegraphics[width=1\textwidth]{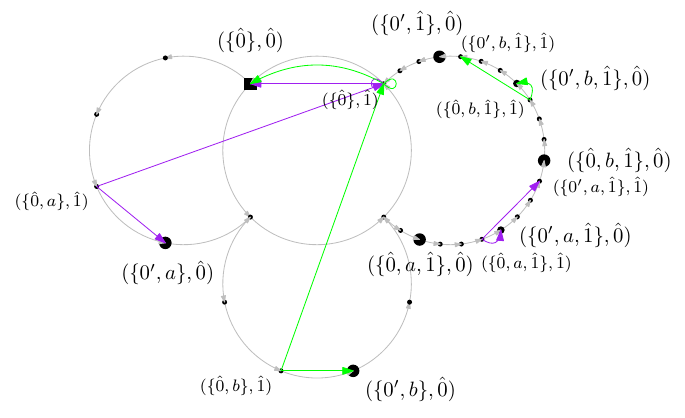}
    \end{minipage}
\caption{Illustration of 
$D[K_{\mathbf{c'}}]$  as a continuation of Example~\ref{example_s}. 
On the left, a drawing of $L$, where the comparable pairs leading to $K_{\mathbf {c'}}$-arcs are highlighted. On the right, $\mathbf{c'}(a,{\hat{1}})$-arcs in purple and $\mathbf{c'}(b,{\hat{1}})$-arcs in green. Note that $\mathcal Q_{\mathbf{c'}(a,{\hat{1}})}=\{\{{\hat{0}}\},\{0',a\},\{0',a,{\hat{1}}\}\}$.}\label{fig:example_cc}
\end{figure}

\begin{remark}\label{rem:A_c'-arcs} Let $x,y\in L$ with $\hat 0<x<y$. Then $\Delta^+_{\mathbf{c'}(x,y)}(D)=2$ and $\Delta^-_{\mathbf{c'}(x,y)}(D)\leq 2$.
\end{remark}
\begin{proof} 
By definition every $\mathbf{c'}(x,y)$-arc shares its origin with a unique other $\mathbf{c'}(x,y)$-arc. Suppose that two different $\mathbf{c'}(x,y)$-arcs, say $((Q\setminus\{0'\}\cup\{\hat 0\},y),([Q]_{\hat 1},\hat 1))$ and $((Q'\setminus\{0'\}\cup\{\hat 0\},y),([Q']_{\hat 1},\hat 1))$, have the same endpoint. Here $Q,Q'\in\mathcal Q_{\mathbf{c'}(x,y)}$, $[Q]_{\hat 1}=[Q']_{\hat 1}$ and $Q\neq Q'$. Without loss of generality, we can assume that $[Q]_{\hat 1}\neq Q$. Then $|Q|=2$ and $[Q]_{\hat 1}=\{\hat{0}\}$. Therefore $Q=\{0',x\}$ and $Q'=\{{\hat{0}}\}$.
\end{proof}

\paragraph{The $K_{\mathbf{h}}$-arcs:} To every pair $(x,y)\in L\times L$ with $\hat{0}<x<y$ we associate a color $\mathbf h(x,y)$ in an arbitrary bijective manner. The set of $\mathbf h(x,y)$-arcs is defined as 
\[A_{\mathbf h(x,y)}=\{((\{{\hat{0}},y\},y),(\{{\hat{0}},x,y\},{\hat{0}})),
((\{{\hat{0}}\},y),(\{{\hat{0}}\},{\hat{0}}))\}.\]

\begin{figure}[h]
\centering
\begin{minipage}{0.15\textwidth}
        \centering
        \includegraphics[width=\textwidth]{example_lattice2c}
    \end{minipage}\hfill
    \begin{minipage}{0.7\textwidth}
        \centering
        \includegraphics[width=\textwidth]{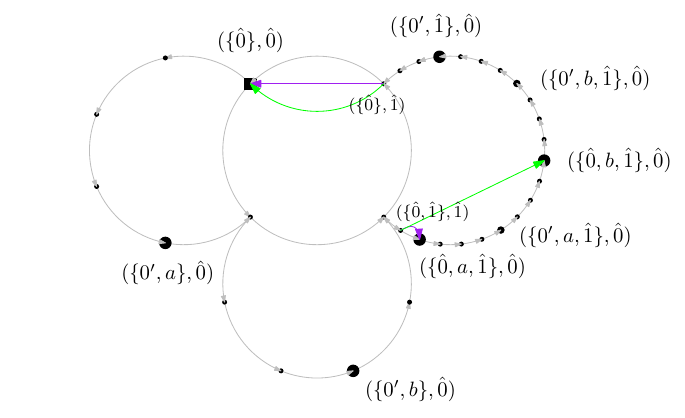}
    \end{minipage}
\caption{Illustration of $D[K_{\mathbf{h}}
]$ as a continuation of Example~\ref{example_s}. 
On the left, a drawing of $L$, where the comparable pairs leading to $K_{\mathbf{h}}$-arcs are highlighted. On the right, $\mathbf{h}(a,{\hat{1}})$-arcs in purple and $\mathbf{h}(b,{\hat{1}})$-arcs in green.}\label{fig:example_h}
\end{figure}

\begin{remark}\label{rem:A_h-arcs} Let $x,y\in L$ with $\hat 0<x<y$. Then $\Delta^+_{\mathbf h(x,y)}(D)=\Delta^-_{\mathbf h(x,y)}(D)=1$.
\end{remark}

\paragraph{The $K_{\mathbf i}$-arcs:} Let us now define the set of $K_{\mathbf i}$-arcs of $D$. The notation stands for ``inheritance": these arcs will ensure that, for $Q=\{Q_1,...,Q_k\},Q'=\{Q'_1,\ldots,Q'_{k'}\}\in\mathcal Q'_{L^+}$ with $k<k'$ and $Q_1=Q'_1,...,Q_k=Q'_k$, the endomorphic images of $W_Q$ and 
$W_{Q'}$ are related. To every pair $(x,y)\in L\times L$ with $\hat{0}<x<y$ we associate a color $\mathbf i(x,y)$ in an arbitrary bijective manner. Let $\mathcal Q_{\mathbf i(x,y)}=\{Q\in\mathcal Q'_{L^+} \mid |Q|\geq 3,\ Q_1={\hat{0}},\ Q^2=x,\ Q^1=y\}$. The set of $\mathbf i(x,y)$-arcs is defined as
{\footnotesize 
\[A_{\mathbf i(x,y)}=
\{((Q,y),(Q\setminus\{y\},y)),((Q\setminus\{y\},y),(Q\setminus\{y\},y)),((Q\setminus\{x,y\},y),(Q\setminus\{x,y\},y)) \mid Q\in\mathcal Q_{\mathbf i(x,y)}\}.
\]}

\begin{figure}[h]
\centering
\begin{minipage}{0.15\textwidth}
        \centering
        \includegraphics[width=1\textwidth]{example_lattice2c}
    \end{minipage}\hfill
    \begin{minipage}{0.7\textwidth}
        \centering
        \includegraphics[width=1\textwidth]{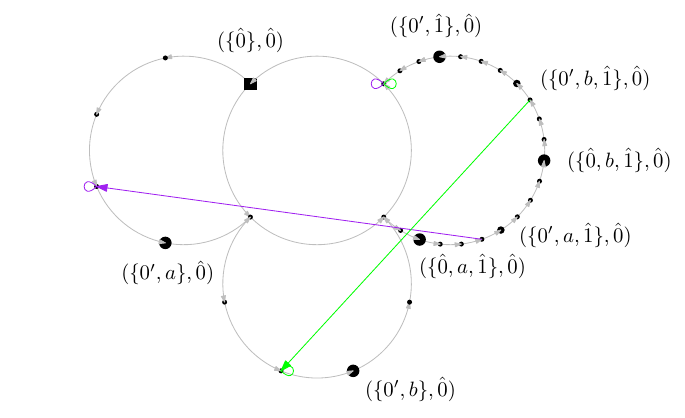}
    \end{minipage}
\caption{Illustration of $D[K_{\mathbf{i}}
]$  as a continuation of Example~\ref{example_s}. On the left, a drawing of $L$, where the comparable pairs leading to $K_{\mathbf i}$-arcs are highlighted. On the right,  $\mathbf{i}(a,{\hat{1}})$-arcs in purple and $\mathbf{i}(b,{\hat{1}})$-arcs in green. Note that $\mathcal Q_{\mathbf i(a,{\hat{1}})}=\{\{{\hat{0}},a,{\hat{1}}\}\}$.}\label{fig:example_i}
\end{figure}

\begin{remark}\label{rem:A_i-arcs} Let $x,y\in L$ with $\hat 0<x<y$. Then $\Delta^+_{\mathbf i(x,y)}(D)\leq 1$ and $\Delta^-_{\mathbf i(x,y)}(D)\leq 2$.
\end{remark}
\begin{proof} Let $u$ be a vertex of $D$. It is clear that, if $u$ is adjacent to a $K_{\mathbf i}$-arc, then it rises from a unique $Q\in\mathcal Q_{\mathbf i(x,y)}$ in the definition of $A_{\mathbf i(x,y)}$ as it is formulated. Therefore,
\[\deg^+_{\mathbf i(x,y)}(u)=\begin{cases}
1 &\text{if } \exists Q\in\mathcal Q_{\mathbf i(x,y)}\ u=(Q,y) \text{ or } u=(Q\setminus\{y\},y) \text{ or } u=(Q\setminus\{x,y\},y) \\
0 &\text{otherwise}
\end{cases}\]
and
\[\deg^-_{\mathbf i(x,y)}(u)=\begin{cases}
2 &\text{if } \exists Q\in\mathcal Q_{\mathbf i(x,y)}\ u=(Q\setminus\{y\},y) \\
1 &\text{if } \exists Q\in\mathcal Q_{\mathbf i(y,x)}\ u=(Q\setminus\{x,y\},y) \\
0 &\text{otherwise.}
\end{cases}\]
\end{proof}

\paragraph{The $K_{\mathbf j}$-arcs:} Finally we define the set of $K_{\mathbf j}$-arcs of $D$. The notation stands for ``join": with these arcs we are going to make sure that, if $x,y\in L$ are $\leq$-incomparable, certain common endomorphic behaviour of the petals $W_{\{x\}},W_{\{y\}}$ is adopted also by $W_{\{x\vee y\}}$. To every pair $(x,y)\in L\times L$ with $x\parallel_{(L,\leq)} y$ and $x\leq^*y$ we associate a color $\mathbf j(x,y)$ in an arbitrary bijective manner. Let $z=x\vee y$. The set of $\mathbf j(x,y)$-arcs is defined as
{\footnotesize\[A_{\mathbf j(x,y)}= 
\{((\{{\hat{0}},x,z\},z),(\{{\hat{0}},y,z\},z)), ((\{{\hat{0}},x\},z),(\{{\hat{0}}\},z)), ((\{{\hat{0}}\},z),(\{{\hat{0}},y\},z)), ((\{{\hat{0}}\},z),(\{{\hat{0}}\},z))\}.\]}

\begin{figure}[h]
\centering
\begin{minipage}{0.15\textwidth}
        \centering
        \includegraphics[width=1\textwidth]{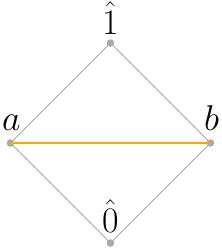}
    \end{minipage}\hfill
    \begin{minipage}{0.7\textwidth}
        \centering
        \includegraphics[width=1\textwidth]{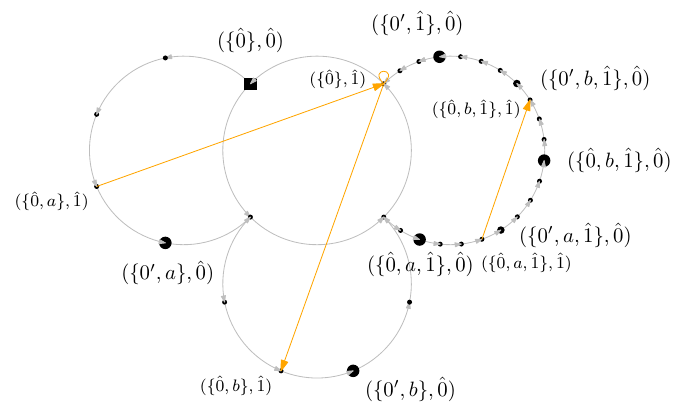}
    \end{minipage}
\caption{Illustration of $D[K_{\mathbf{j}}
]$ as a continuation of Example~\ref{example_s}. On the left, a drawing of $L$, where the unique incomparable pair leading to $K_{\mathbf j}$-arcs is highlighted. On the right, $\mathbf{j}(a,b)$-arcs in orange.}\label{fig:example_j}
\end{figure}

\begin{remark}\label{rem:A_j-arcs} Let $x,y\in L$ be two $\leq$-incomparable elements with $x\leq^*y$. Then $\Delta^+_{\mathbf j(x,y)}(D)=\Delta^-_{\mathbf j(x,y)}(D)\leq 2$.
\end{remark}
\begin{proof} Let $z=x\vee y$
, and let $u$ be a vertex of $D$. It is clear that
\[\deg^+_{\mathbf j(x,y)}(u)=\begin{cases}2 &\text{if } u=(\{{\hat{0}}\},z)\\ 1 &\text{if } u=(\{{\hat{0}},x,z\},z) \text{ or } u=(\{{\hat{0}},x\},z) \\ 0 &\text{otherwise}\end{cases}\]
and that
\[\deg^-_{\mathbf j(x,y)}(u)=\begin{cases}2 &\text{if } u=(\{{\hat{0}}\},z)\\ 1 &\text{if } u=(\{{\hat{0}},y,z\},z) \text{ or } u=(\{{\hat{0}},y\},z) \\ 0 &\text{otherwise.}\end{cases}\]
\end{proof}

\subsubsection{Description of $\End(D)$}
Having defined the arc-colored digraph $D$, in the current subsection we will show that its endomorphism monoid coincides with $L$.

\begin{defi}\label{defi:varphi_x} Let $w\in L$. Define $\varphi_w:V\rightarrow V$ as \[\varphi_w(([Q]_x,x))=\begin{cases}([Q\,\cap\!\downarrow_{L^+}\! w]_x,x) & \text{if } Q\,\cap\!\downarrow_{L^+}\!w\neq\{0'\}\\ (\{{\hat{0}}\},x) &\text{otherwise}\end{cases}\] for every $Q\in\mathcal Q'_{L^+}$ and every $x\in L$.
\end{defi}

It is easy to verify that $\varphi_w$ is well-defined
. Indeed, suppose that $[Q]_x=[Q']_x$ but $Q\neq Q'$. Recall that in this situation $[Q]_x=[Q']_x=\{\hat 0\}$. Either $Q$ has two elements or $Q=\{\hat 0\}$, and the same holds for $Q'$. Hence we deduce that $\varphi_w(([Q]_x,x))=(\{\hat 0\},x)=\varphi_w(([Q']_x,x))$. Moreover, it is clear that the image of $\varphi_w$ is included in $V$. The following sequence of lemmas establishes that $\varphi_w$ indeed defines an endomorphism of $D$ for every $w\in L$. 

\begin{lemma}\label{lem:compatibility_chi'_S} Let $w\in L$. Then $\varphi_w\in\End(D[K_{\mathbf s}])$.
\end{lemma}
\begin{proof} Let $x=L_j\in L\setminus\{{\hat{1}}\}$ and $y=L_{j+1}$. Let $a$ be an $\mathbf s(x)$-arc. Then, we have that $a=(([Q]_x,x),([Q]_y,y))$ for some $Q\in\mathcal Q'_{L^+}$. Let $P=Q\,\cap\!\downarrow_{L^+}\! w$. If $P=\{0'\}$, it is clear that $\varphi_w(a)=((\{{\hat{0}}\},x),(\{{\hat{0}}\},y))$ is an $\mathbf s(x)$-arc. If $P\neq\{0'\}$, then $P\in\mathcal Q'_{L^+}$, and it is clear that $\varphi_w(a)=(([P]_x,x),([P]_y,y))$ is an $\mathbf s(x)$-arc. 

Now let $a$ be an $\mathbf s({\hat{1}})$-arc. Then $a=(([Q]_{\hat 1},\hat 1),(\tilde Q,{\hat{0}}))$ for some $Q\in\mathcal Q'_{L^+}$. Let $P=Q\,\cap\!\downarrow_{L^+}\! w$, and let $k$ be the largest integer such that $Q_k\leq w$. We proceed with an analysis of the different possible cases. 

\setcounter{case}{0}
\begin{case}
$Q_1={\hat{0}}$ and $|P|\geq 2$.
\end{case} 
Here $k\geq 2$ and 
\[\tilde Q\,\cap\!\downarrow_{L^+}\! w=\begin{cases}\{{\hat{0}},(\downarrow_ L Q_2)_2,Q_2,...,Q_k\} &\text{if } (\downarrow_ L Q_2)_2<Q_2 \\ \{0',Q_2,...,Q_k\} &\text{if }  (\downarrow_ L Q_2)_2=Q_2. \end{cases} \]
Therefore, $[\tilde Q\,\cap\!\downarrow_{L^+}\! w]_{\hat{0}}=\tilde Q\,\cap\!\downarrow_{L^+}\! w=\tilde P$, so $\varphi_w(a)=(([P]_{\hat 1},\hat 1),([\tilde Q\,\cap\!\downarrow_{L^+}\! w]_{\hat{0}},{\hat{0}}))=(([P]_{\hat 1},\hat 1),(\tilde P,{\hat{0}}))$, which is an $\mathbf s({\hat{1}})$-arc.

\begin{case} $Q_1={\hat{0}}$ and $P=\{{\hat{0}}\}$.
\end{case}
We distinguish three subcases. 
\begin{itemize}
   \item[] \textit{Subcase (i):} $Q=\{{\hat{0}}\}$. Then it is immediate that $\varphi_w(a)$ is an $\mathbf s({\hat{1}})$-arc. 
   \item[] \textit{Subcase (ii):} $|Q|\geq 2$ and $(\downarrow_ L Q_2)_2<Q_2$. Then $\tilde Q\,\cap\!\downarrow_{L^+}\! w$ is either $\{{\hat{0}}\}$ or $\{{\hat{0}},(\downarrow_ L Q_2)_2\}$. In any case, $[\tilde Q\,\cap\!\downarrow_{L^+}\! w]_{\hat{0}}=\{{\hat{0}}\}=\tilde P$, so we obtain the same conclusion of Case~1. 
   \item[] \textit{Subcase (iii):} $|Q|\geq 2$ and $(\downarrow_ L Q_2)_2=Q_2$. Then $\tilde Q\,\cap\!\downarrow_{L^+}\! w=\{0'\}$. Thus $\varphi_w(a)=((\{{\hat{0}}\},{\hat{1}}),(\{{\hat{0}}\},{\hat{0}}))$, which is an $\mathbf s({\hat{1}})$-arc.
\end{itemize}

\begin{case} $Q_1=0'$ and $|P|\geq 3$.
\end{case}
Here $k\geq 3$ and 
\[\tilde Q\,\cap\!\downarrow_{L^+}\! w=\begin{cases}\{{\hat{0}},(\downarrow_ L Q_3)_{i+1},Q_3,...,Q_k\} &\text{if } (\downarrow_ L Q_3)_{i+1}<Q_3 \\ \{0',Q_3,...,Q_k\} &\text{if }  (\downarrow_ L Q_3)_{i+1}=Q_3, \end{cases} \]
where $i$ is the positive integer with $Q_2=(\downarrow_ L Q_3)_{i}$. Hence $[\tilde Q\,\cap\!\downarrow_{L^+}\! w]_{\hat{0}}=\tilde Q\,\cap\!\downarrow_{L^+}\! w=\tilde P$, and we end as in Case~1.

\begin{case} $Q_1=0'$ and $|P|=2$.
\end{case} 
Let $i$ be the positive integer with $Q_2=(\downarrow_ L Q_3)_{i}$ as before. Distinguish three subcases. 
\begin{itemize}
    \item[] \textit{Subcase (i):} $|Q|=2$. Then $P=Q$ and $\{{\hat{0}}\}=\tilde P=\tilde Q=\tilde Q\,\cap\!\downarrow_{L^+}\! w=[\tilde Q\,\cap\!\downarrow_{L^+}\! w]_{\hat{0}}$, so $\varphi_w(a)=a$. 
    \item[] \textit{Subcase (ii):} $|Q|\geq 3$ and $(\downarrow_ L Q_3)_{i+1}<Q_3$. We have that $\tilde Q\,\cap\!\downarrow_{L^+}\! w$ is either $\{{\hat{0}}\}$ or $\{{\hat{0}},(\downarrow_ L Q_3)_{i+1}\}$. In any case, $[\tilde Q\,\cap\!\downarrow_{L^+}\! w]_{\hat{0}}=\{{\hat{0}}\}=\tilde P$, and we end as in Case~1.
    \item[] \textit{Subcase (iii):} $|Q|\geq 3$ and $(\downarrow_ L Q_3)_{i+1}=Q_3$. Then $\tilde Q\,\cap\!\downarrow_{L^+}\! w=\{0'\}$. Thus $\varphi_w(a)=((\{{\hat{0}}\},{\hat{1}}),(\{{\hat{0}}\},{\hat{0}}))$, which is an $\mathbf s({\hat{1}})$-arc.
\end{itemize}

\begin{case} $Q_1=0'$ and $P=\{0'\}$.
\end{case}
We have that 
\[\tilde Q\,\cap\!\downarrow_{L^+}\! w=\begin{cases}\{{\hat{0}}\} &\text{if } |Q|=2 \\ Q\cup\{{\hat{0}},(\downarrow_ L Q_3)_{i+1}\}\setminus\{0',Q_2\}\,\cap\!\downarrow_{L^+}\! w &\text{if } |Q|\geq 3 \text{ and } (\downarrow_ L Q_3)_{i+1}<Q_3 \\ Q\setminus\{Q_2\}\,\cap\!\downarrow_{L^+}\! w & \text{if } |Q|\geq 3 \text{ and } (\downarrow_ L Q_3)_{i+1}=Q_3, \end{cases}\]
where $i$ is such that $Q_2=(\downarrow_ L Q_3)_{i}$. Therefore, either $[\tilde Q\,\cap\!\downarrow_{L^+}\! w]_{\hat{0}}=\{{\hat{0}}\}$ or $\tilde Q\,\cap\!\downarrow_{L^+}\! w=\{0'\}$. In any case, $\varphi_w(a)=((\{{\hat{0}}\},{\hat{1}}),(\{{\hat{0}}\},{\hat{0}}))$, which is an $\mathbf s({\hat{1}})$-arc.
\end{proof}

\begin{lemma}\label{lem:compatibility_chi'_R} Let $w\in L$. Then $\varphi_w\in\End(D[ K_{\mathbf r}])$.
\end{lemma}
\begin{proof} Recall that all $ K_{\mathbf r}$-arcs are loops. Let $x\in L\setminus\{{\hat{0}}\}$, and let $([Q]_y,y)$ be a vertex of $D$ with an $\mathbf r(x)$-loop, where $y\in L$ and $Q\in\mathcal Q_{\mathbf r(x)}$. Since $\{{\hat{0}}\}\in\mathcal Q_{\mathbf r(x)}$ and, if $Q\,\cap\!\downarrow_{L^+}\!w\neq\{0'\}$, also $Q\,\cap\!\downarrow_{L^+}\!w\in\mathcal Q_{\mathbf r (x)}$, we have that $\varphi_w(([Q]_y,y))$ has an $\mathbf r(x)$-loop.
\end{proof}

\begin{lemma}\label{lem:compatibility_K'_C} Let $w\in L$. Then $\varphi_w\in\End(D[K_{\mathbf c}])$.
\end{lemma}
\begin{proof} Let $x,y\in L$ with ${\hat{0}}<x<y$. Any $\mathbf c(x,y)$-arc $a$ of $D$ is of the form $((Q,y),([Q]_{\hat{0}},{\hat{0}}))$ for some $Q\in\mathcal Q_{\mathbf c(x,y)}$. Since $\{{\hat{0}}\}\in\mathcal Q_{\mathbf c(x,y)}$ and, if $Q\,\cap\!\downarrow_{L^+}\!w\neq\{0'\}$, also $Q\,\cap\!\downarrow_{L^+}\!w\in\mathcal Q_{\mathbf c(x,y)}$, we have that $\varphi_w(a)$ is a $\mathbf c(x,y)$-arc.
\end{proof}

\begin{lemma}\label{lem:compatibility_K'_{C'}} Let $w\in L$. Then $\varphi_w\in\End(D[K_{\mathbf{c'}}])$.
\end{lemma}
\begin{proof} Let $x,y\in L$ with ${\hat{0}}<x<y$. Any $\mathbf{c'}(x,y)$-arc $a$ of $D$ is of the form $((Q\setminus\{0'\}\cup\{{\hat{0}}\},y),(Q,{\hat{0}}))$ or $((Q\setminus\{0'\}\cup\{{\hat{0}}\},y),([Q]_{\hat{1}},{\hat{1}}))$ for some $Q\in\mathcal Q_{\mathbf{c'}(x,y)}$. Since $\{{\hat{0}}\}\in\mathcal Q_{\mathbf{c'}(x,y)}$ and, if $Q\,\cap\!\downarrow_{L^+}\!w\neq\{0'\}$, also $Q\,\cap\!\downarrow_{L^+}\!w\in\mathcal Q_{\mathbf{c'}(x,y)}$, we have that $\varphi_w(a)$ is a $\mathbf{c'}(x,y)$-arc.
\end{proof}

\begin{lemma}\label{lem:compatibility_K'_H} Let $w\in L$. Then $\varphi_w\in\End(D[K_{\mathbf h}])$.
\end{lemma}
\begin{proof} Let $x,y\in L$ with ${\hat{0}}<x<y$. We see that the image by $\varphi_w$ of the $\mathbf h(x,y)$-arc $((\{{\hat{0}},y\},y),(\{{\hat{0}},x,y\},{\hat{0}}))$ is either itself or $((\{{\hat{0}}\},y),(\{{\hat{0}}\},{\hat{0}}))$, which is also an $\mathbf h(x,y)$-arc. Moreover, $((\{{\hat{0}}\},y),(\{{\hat{0}}\},{\hat{0}}))$ is fixed by $\varphi_w$. 
\end{proof} 

\begin{lemma}\label{lem:compatibility_chi'_I} Let $w\in L$. Then $\varphi_w\in\End(D[ K_{\mathbf i}])$.
\end{lemma}
\begin{proof} Let $x,y\in L$ with $\hat 0<x<y$, and let $u,v$ be two vertices of $D$ such that $(u,v)$ is a $\mathbf i(x,y)$-arc. Distinguish three cases.

\setcounter{case}{0}
\begin{case} $u=(Q,y)$ and $v=(Q\setminus\{y\},y)$ for some $Q\in\mathcal Q_{\mathbf i(x,y)}$.
\end{case}
Let $i$ be the greatest index such that $Q_i\leq w$. If $i=|Q|$ then $\varphi_w(u)=u$ and $\varphi_w(v)=v$. If $i\leq|Q|-1$ then $\varphi_w(u)=\varphi_w(v)=(\{Q_1,...,Q_i\},y)$, and we are done, because $\{Q_1,...,Q_i,x,y\}\in\mathcal Q_{\mathbf i(x,y)}$.

\begin{case} $u=v=(Q\setminus\{y\},y)$ for some $Q\in\mathcal Q_{\mathbf i(x,y)}$.
\end{case}
Let $i$ be the greatest index such that $Q_i\leq w,x$. Then $\varphi_w(u)=(\{Q_1,...,Q_i\},y)$, and we are done, because $\{Q_1,...,Q_i,x,y\}\in\mathcal Q_{\mathbf i(x,y)}$.

\begin{case} $u=v=(Q\setminus\{x,y\},y)$ for some $Q\in\mathcal Q_{\mathbf i(x,y)}$.
\end{case}
Let $i$ be the greatest index such that $Q_i\leq w,Q^3$. Then $\varphi_w(u)=(\{Q_1,...,Q_i\},y)$, and we are done, because $\{Q_1,...,Q_i,x,y\}\in\mathcal Q_{\mathbf i(x,y)}$.
\end{proof}

\begin{lemma}\label{lem:compatibility_chi'_J} Let $w\in L$. Then $\varphi_w\in\End(D[ K_{\mathbf j}])$.
\end{lemma}
\begin{proof} Let $x,y\in L$ be two $\leq$-incomparable elements with $x<^*y$, and let $u,v$ be two vertices of $D$ such that $(u,v)$ is a $\mathbf j(x,y)$-arc. Let $z=x\vee y$. We distinguish four cases.

\setcounter{case}{0}
\begin{case} $u=(\{{\hat{0}},x,z\},z)$ and $v=(\{{\hat{0}},y,z\},z)$.
\end{case}
If $x,y\not\leq w$ then $\varphi_w(u)=\varphi_w(v)=(\{{\hat{0}}\},z)$, and we are done. If $x\leq w$ and $y\not\leq w$ then $\varphi_w(u)=(\{{\hat{0}},x\},z)$, $\varphi_w(v)=(\{{\hat{0}}\},z)$, and we are done. If $x\not\leq w$ and $y\leq w$ a symmetric situation occurs. If $x,y\leq w$ then also $z\leq w$, so $u,v$ are fixed under $\varphi_w$.

\begin{case}$u=(\{{\hat{0}},x\},z)$ and $v=(\{{\hat{0}}\},z)$.
\end{case}
If $x\not\leq w$ then $\varphi_w(u)=\varphi_w(v)=(\{{\hat{0}}\},z)$, and we are done. If $x\leq w$ then $u,v$ are fixed under $\varphi_w$.

\begin{case} $u=(\{{\hat{0}}\},z)$ and $v=(\{{\hat{0}},y\},z)$.
\end{case}
A symmetric argument to that of Case~2 applies.

\begin{case} $u=v=(\{{\hat{0}}\},z)$.
\end{case}
In this case $(u,v)$ is fixed by $\varphi_w$.
\end{proof}

Lemmas~\ref{lem:compatibility_chi'_S},~\ref{lem:compatibility_chi'_R},~\ref{lem:compatibility_K'_C},~\ref{lem:compatibility_K'_{C'}},~\ref{lem:compatibility_K'_H},~\ref{lem:compatibility_chi'_I}
and \ref{lem:compatibility_chi'_J} show that $\varphi_w$ maps colored arcs to arcs of the same color. Thus, we have shown:
\begin{prop}\label{prop:Phi_well_defined} Let $w\in L$. Then $\varphi_w\in\End(D)$.
\end{prop}

We now show that every endomorphism of $D$ is of the form $\varphi_w$ for some $w\in L$. 
In order to do so for every  $\varphi\in\End(D)$ we define the set $L_{\varphi}=\{x\in L\mid \varphi(W_{\{x\}})=W_{\{x\}}\}$. The following sequence of lemmas  will establish that $L_{\varphi}$ is a principal ideal and hence we will be able to associate its maximal element $\ell_{\varphi}$ to $\varphi$. Then we will show that every $\varphi\in\End(D)$ is precisely $\varphi_{\ell_{\varphi}}$.

\begin{lemma}\label{lem:base_cycle_fixed} For any $\varphi\in\End(D)$, ${\hat{0}}\in L_{\varphi}$.
\end{lemma}
\begin{proof} By examining $D[ K_{\mathbf s}]$ one realizes that $W_{\{{\hat{0}}\}}$ is its only directed closed walk of length $n$ starting with an $\mathbf s({\hat{0}})$-arc. In fact, this is the shortest possible. Therefore, $\varphi(W_{\{{\hat{0}}\}})=W_{\{{\hat{0}}\}}$.
\end{proof}

\begin{lemma}\label{lem:easy_lemma} Let $\varphi\in\End(D)$, $x\in L$, and $v$ be a vertex of $W_{\{x\}}$ which is not a vertex of $W_{\{{\hat{0}}\}}$. If $\varphi(v)=v$ then $x\in L_{\varphi}$. 
\end{lemma}
\begin{proof} Note that there is a unique shortest $K_{\mathbf s}$-walk from $(\{{\hat{0}}\},{\hat{0}})$ to $v$ and a unique shortest $K_{\mathbf s}$-walk from $v$ to $(\{{\hat{0}}\},{\hat{0}})$; name them $W_1$ and $W_2$. From the fact that $v$ is not a vertex of $W_{\{{\hat{0}}\}}$ we get that $W_{\{x\}}=W_1W_2$. Therefore $\varphi(W_{\{x\}})=\varphi(W_1W_2)=\varphi(W_1)\varphi(W_2)=W_1W_2=W_{\{x\}}$, using that $\varphi(v)=v$ and that $\varphi((\{{\hat{0}}\},{\hat{0}}))=(\{{\hat{0}}\},{\hat{0}})$ by Lemma~\ref{lem:base_cycle_fixed}.
\end{proof}

\begin{lemma}\label{lem:petal_periphery} Let $\varphi\in\End(D)$ and $x\in L$. If $x\notin L_{\varphi}$, then $\varphi(W_{\{{\hat{0}},x\}})=\varphi(W_{\{0',x\}})=W_{\{{\hat{0}}\}}$.
\end{lemma}
\begin{proof} Suppose that $\varphi(W_{\{{\hat{0}},x\}})\neq W_{\{{\hat{0}}\}}$. Then, since $x\notin L_{\varphi}$, by Lemma~\ref{lem:easy_lemma} there is some $y\in L$ with $y>^*x$ such that $\varphi((\{{\hat{0}},x\},y))=(\{{\hat{0}},y\},y)$. But $(\{{\hat{0}},x\},y)$ has an $\mathbf r(y)$-loop, and $(\{{\hat{0}},y\},y)$ does not by Lemma~\ref{lem:anticolors}. The remainder is similar: suppose that $\varphi(W_{\{0',x\}})\neq W_{\{{\hat{0}}\}}$. Since $x\notin L_{\varphi}$, there are some $\leq^*$-consecutive $w,w'\in L$ with $x>^*w>^*w'$ such that $\varphi((\{0',x\},w'))=(\{0',w\},w')$. But $(\{0',x\},w')$ has a $\mathbf r(w)$-loop and $(\{0',w\},w')$ does not.
\end{proof}

\begin{lemma}\label{lem:partial_petal_folding} Let $\varphi\in\End(D)$ and $x,y\in L$ with $x<y$. If $y\notin L_{\varphi}$ and $x\in L_{\varphi}$, then $\varphi(W_{\{x,y\}})=W_{\{x\}}$.
\end{lemma}
\begin{proof} If $x={\hat{0}}$, this is a consequence of Lemma~\ref{lem:petal_periphery}. Therefore, assume that $x>{\hat{0}}$. By Lemma~\ref{lem:petal_periphery} $\varphi((\{{\hat{0}},y\},y))=(\{{\hat{0}}\},y)$. Since there is only one $\mathbf h(x,y)$-arc with origin $(\{{\hat{0}}\},y)$ we have that the image of the first vertex of $W_{\{x,y\}}$ is $\varphi((\{{\hat{0}},x,y\},{\hat{0}}))=(\{{\hat{0}}\},{\hat{0}})$. By Lemma~\ref{lem:easy_lemma} the image by $\varphi$ of the $\mathbf i(x,y)$-arc $((\{{\hat{0}},x,y\},y),(\{{\hat{0}},x\},y))$ is $((\{{\hat{0}},x\},y),(\{{\hat{0}},x\},y))$. To finish the argument we also need to know the image of the last vertex of $W_{\{x,y\}}$. Let $i$ be the integer with $(\downarrow_L y)_i=x$. 
Distinguish two cases.

\setcounter{case}{0}
\begin{case} $(\downarrow_L y)_{i+1}<y$.
\end{case} 
In that case the last vertex of $W_{\{x,y\}}$ is $(\{{\hat{0}},(\downarrow_L y)_{i+1},y\},{\hat{0}})$, and we get $\varphi((\{{\hat{0}},(\downarrow_L y)_{i+1},y\},{\hat{0}}))=(\{{\hat{0}}\},{\hat{0}})$ arguing as above. 

\begin{case} $(\downarrow_L y)_{i+1}=y$.
\end{case} 
In that case the last vertex of $W_{\{x,y\}}$ is $(\{0',y\},{\hat{0}})$, and we get $\varphi((\{0',y\},{\hat{0}}))=(\{{\hat{0}}\},{\hat{0}})$ by Lemma~\ref{lem:petal_periphery}. 

From the images of $(\{\hat 0,x,y\},y)$ and the first and last vertices of $W_{\{x,y\}}$, it is deduced that $\varphi(W_{\{x,y\}})=W_{\{x\}}$. 
\end{proof}

\begin{prop} For every $\varphi\in\End(D)$, there is an $\ell_{\varphi}\in L$ such that  $L_{\varphi}=\downarrow_L\ell_{\varphi}$ is the principal ideal of $L$ with maximal element $\ell_{\varphi}$.
\end{prop}
\begin{proof} To see that $L_{\varphi}$ is an ideal, pick $x,y\in L$ with $x<y$ and $y\in L_{\varphi}$. Since $(\{{\hat{0}},x,y\},y)$ is fixed by $\varphi$, the $\mathbf i(x,y)$-arc $((\{{\hat{0}},x,y\},y),(\{{\hat{0}},x\},y))$ is fixed by $\varphi$. By Lemma~\ref{lem:easy_lemma}, $x\in L_{\varphi}$.

To see that $L_{\varphi}$ is principal, let $x,y\in L_{\varphi}$ be two $\leq$-incomparable elements with $x<^*y$. Suppose for a contradiction that $z=x\vee y\notin L_{\varphi}$. By Lemma~\ref{lem:partial_petal_folding} the image by $\varphi$ of the $\mathbf j(x,y)$-arc $((\{{\hat{0}},x,z\},z),(\{{\hat{0}},y,z\},z))$ is $((\{{\hat{0}},x\},z),(\{{\hat{0}},y\},z))$, but this is not a $\mathbf j(x,y)$-arc.

We can define $\ell_{\varphi}=\max L_{\varphi}$.
\end{proof}

\begin{prop}\label{prop:Phi_surjective} Let $\varphi\in\End(D)$. Then $\varphi=\varphi_{\ell_{\varphi}}$.
\end{prop}
\begin{proof} Let $Q\in\mathcal Q'_{L^+}$ and $y\in L$; we will see that $\varphi(([Q]_y,y))=\varphi_{\ell_{\varphi}}(([Q]_y,y))$. Distinguish three cases.

\setcounter{case}{0}
\begin{case} $Q^1\in L_{\varphi}$.
\end{case}
Then it is clear that $\varphi(([Q]_y,y))=([Q]_y,y)=([Q\cap\!\downarrow_{L^+}\!\ell_{\varphi}]_y,y)=\varphi_{\ell_{\varphi}}(([Q]_y,y))$.

\begin{case} $Q^1\notin L_{\varphi}$ and $|Q|=2$.
\end{case}
The conclusion is obtained by applying Lemma~\ref{lem:petal_periphery}.

\begin{case} $Q^1\notin L_{\varphi}$ and $|Q|\geq 3$.
\end{case}
Assume that $Q_1={\hat{0}}$ and let $Q'=Q\cup\{0'\}\setminus\{{\hat{0}}\}$. Let $a$ be the $\mathbf i(Q^2,Q^1)$-arc $((Q,Q^1),(Q\setminus\{Q^1\},Q^1))$. By induction, $\varphi((Q\setminus\{Q^1\},Q^1))=(Q\setminus\{Q^1\}\;\cap\!\downarrow_{L^+}\!\ell_{\varphi}, Q^1)$. Let $i$ be the largest integer such that $Q_i\in\;\downarrow_{L^+}\!\ell_{\varphi}$. Note that $\varphi(a)=((\{Q_1,...,Q_i\},Q^1),(\{Q_1,...,Q_i\},Q^1))$ by Lemma~\ref{lem:easy_lemma}. Let $W_1,W_2$ be the shortest $K_{\mathbf s}$-walks from $(Q,{\hat{0}})$ to $(Q,Q^1)$ and from $(Q,Q^1)$ to $(Q',\hat 1)$, respectively. Consider the following subcases (see Figure~\ref{fig:analysis_of_cases}).
\begin{itemize}
    \item[] \textit{Subcase (i):} $i\geq 3$. It is clear that the images of $W_1,W_2$ by $\varphi$ are already determined by the image of $(Q,Q^1)$; in particular, $\varphi((Q,y))=(\{Q_1,...,Q_i\},y)$ and $\varphi((Q',y))=(\{0',Q_2,...,Q_i\},y)$. 
    \item[] \textit{Subcase (ii):} $i=2$. Recall that for every $x\in L$ different from $Q_2$ there is an $\mathbf r(x)$-loop on each vertex of the form $([Q]_z,z)$ or $([Q']_z,z)$, where $z\in L$. Together with Lemma~\ref{lem:anticolors}, this determines $\varphi(W_1),\varphi(W_2)$; in particular, $\varphi((Q,y))=([\{{\hat{0}},Q_2\}]_y,y)$ and $\varphi((Q',y))=([\{0',Q_2\}]_y,y)$. 
    \item[] \textit{Subcase (iii):}  $i=1$. In this case $\varphi((Q,Q^1))=(\{{\hat{0}}\},Q^1)$ implies that the image by $\varphi$ of the $\mathbf c(Q_2,Q^1)$-arc $((Q,Q^1),([Q]_{\hat{0}},{\hat{0}}))$ is $((\{{\hat{0}}\},Q^1),(\{{\hat{0}}\},{\hat{0}}))$. With the argument of the $K_{\mathbf r}$-arcs used in Subcase (ii) we conclude that $\varphi(W_{Q})=W_{\{{\hat{0}}\}}$; in particular $\varphi(([Q]_y,y))=(\{{\hat{0}}\},y)$. On the other hand, the image by $\varphi$ of the $\mathbf{c'}(Q_2,Q^1)$-arcs $((Q,Q^1),(Q',{\hat{0}}))$ and $((Q,Q^1),([Q']_{\hat{1}},{\hat{1}}))$ is $((\{{\hat{0}}\},Q^1),(\{{\hat{0}}\},{\hat{0}}))$ and $((\{{\hat{0}}\},Q^1),(\{{\hat{0}}\},{\hat{1}}))$. Therefore $\varphi(W_{Q'})=W_{\{{\hat{0}}\}}$, and in particular $\varphi(([Q']_y,y))=(\{{\hat{0}}\},y)$.
\end{itemize}
\end{proof}
\begin{figure}[h]
\centering
\includegraphics[scale=1]{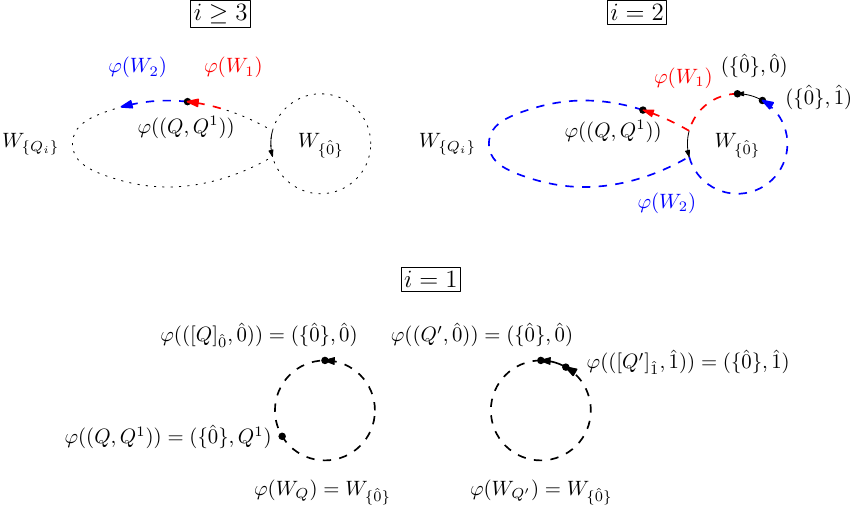}
\caption{The three subcases when $Q^1\notin L_{\varphi}$ and $|Q|\geq 3$ in Proposition~\ref{prop:Phi_surjective}. The solid lines denote actual arcs, the dotted and dashed ones denote walks.}\label{fig:analysis_of_cases}
\end{figure}

We are ready to show that $L$ is isomorphic to the endomorphism monoid of $D$. The next proposition, together with Remarks~\ref{rem:A_s-arcs}, \ref{rem:A_r-arcs}, \ref{rem:A_c-arcs}, \ref{rem:A_c'-arcs}, \ref{rem:A_h-arcs}, \ref{rem:A_i-arcs} and \ref{rem:A_j-arcs}, ends the proof of Lemma~\ref{lem:first_step}.

\begin{prop} The mapping $\Phi:(L,\wedge)\rightarrow\End(D)$ defined by $\Phi(x)=\varphi_x$ is a monoid isomorphism.
\end{prop}
\begin{proof} Note that by Proposition~\ref{prop:Phi_well_defined} the mapping $\Phi$ is well defined. It is clear from Definition~\ref{defi:varphi_x} that it is injective: let $x,y\in L$ with $x<^*y$ and $v=(\{{\hat{0}},y\},y)\in V$; then $\varphi_x(v)\neq v=\varphi_y(v)$. By Proposition~\ref{prop:Phi_surjective}, it is also surjective. Clearly $\Phi({\hat{1}})=\textrm{id}_{D}$. Finally, let $x,y,z\in L$ and $Q\in\mathcal Q'_{L^+}$. Then
\[\varphi_x(\varphi_y(([Q]_z,z)))=\]
{\small
\[\left.\begin{cases}\varphi_x(([Q\,\cap\!\downarrow_{L^+}\! y]_z,z))=\begin{cases}([Q\,\cap\!\downarrow_{L^+}\! x \,\cap\!\downarrow_{L^+}\! y]_z,z) \\ (\{{\hat{0}}\},z)\end{cases} & \!\!\!\!\!\!\!\begin{matrix}\text{if }Q\,\cap\!\downarrow_{L^+}\! x\,\cap\!\downarrow_{L^+}\! y\neq\{0'\} \hfill \\ \text{if } Q\,\cap\!\downarrow_{L^+}\! y\neq\{0'\}= Q\,\cap\!\downarrow_{L^+}\! x\,\cap\!\downarrow_{L^+}\! y\end{matrix}\\ \varphi_x((\{{\hat{0}}\},z))=(\{{\hat{0}}\},z) &\!\!\!\!\!\!\! \text{if } Q\,\cap\!\downarrow_{L^+}\! y=\{0'\} \end{cases}\right\}=\]
}

\[\left.\begin{cases}([Q\,\cap\!\downarrow_{L^+}\! (x\wedge y)]_z,z) &\text{if } Q\,\cap\!\downarrow_{L^+}\! (x\wedge y)\neq\{0'\} \\ (\{{\hat{0}}\},z) & \text{otherwise}\end{cases}\right\}=\varphi_{x\wedge y}(([Q]_z,z)).\]
\end{proof}

One of the main results of the paper follows from Lemmas~\ref{lem:first_step},~\ref{lem:second_step}, and~\ref{lem:third_step}:
\begin{teo}\label{teo:mainboundeddegree}
For every finite lattice $L$ there exists a simple undirected graph $G$ of maximum degree at most $3$ such that $\End(G)\cong L$.
\end{teo}

\subsection{Forcing a minor}\label{subsec:forcingaminor}
The main result of this subsection is that no class excluding a minor can represent all lattices (Theorem~\ref{teo:forcingaminor}). 

Recall the definitions of retracts and in particular Lemma~\ref{lem:End_is_lattice}, which in this section we will exclusively use for an undirected graph $G$.
If $\End(G)\cong L$, then Lemma~\ref{lem:End_is_lattice} justifies the following notation: If $\ell\in L$, then $R(\ell)\in\mathcal{R}_G$ is the corresponding retract of $G$. An element $\ell\in L$ is called \emph{join-irreducible} if $\bigvee I=\ell$ implies $\ell\in I$ for all $I\subseteq L$. Denote by $\mathcal{J}(L)$ the subposet of $L$ induced by its join-irreducible elements. Note that an element of $L$ is join-irreducible if and only if it covers a unique element. An elementary property that we will sometimes use is $\ell=\bigvee\{x\in \mathcal{J}(L)\mid x\leq \ell\}$  for all $\ell\in L$.

\begin{lema}\label{lem:nonempty_connected}

 Let $L$ be a lattice and $G$ a graph such that $\End(G)\cong L$. 
 Let $y \in\mathcal{J}(L)$ and $C=\{x_1, \ldots, x_k\}\subseteq\mathcal{J}(L)$ the elements covered by $ y$ in $\mathcal{J}(L)$. We have that $R( y)\setminus R(\bigvee C)$ is non-empty and induces a connected subgraph of $G$.   
\end{lema}
\begin{proof}
 Since $x_1, \ldots, x_k\leq y$, we have $\bigvee C\leq y$. Since $y$ is join-irreducible and $y\notin C$, we have $\bigvee C\neq y$. Moreover, we have that if $\bigvee C\leq z\leq y$, since $z$ is the join of all join-irreducibles below it, if $z\neq y$ then $z=\bigvee C$. We conclude $\bigvee C\prec y$. Note that the above arguments also hold if $C=\varnothing$.
 
 By Lemma~\ref{lem:End_is_lattice} this implies $R(y)\setminus R(\bigvee C)\neq \varnothing$. 
 Suppose now that $R( y)\setminus R(\bigvee C)$ has a connected component $A$ such that $B=(R( y)\setminus R(\bigvee C))\setminus A\neq \varnothing$. Consider the retraction $\varphi:R( y)\to R(\bigvee C)$ and define $\varphi':R( y)\to R(\bigvee C)\cup B$ as 
 $$v\mapsto\begin{cases}
            \varphi(v) & \text{if } v\in A,\\
            v & \text{otherwise.}
\end{cases}$$
Since there are no edges from $A$ to $B$, $\varphi'$ is a retraction from $R(y)$ to $R(\bigvee C)\cup B$. Hence, its image $R(\bigvee C)\cup B$ is a retract of $G$. However, $R(\bigvee C)\subsetneq R(\bigvee C)\cup B\subsetneq R( y)$ together with Lemma~\ref{lem:End_is_lattice} contradicts $\bigvee C\prec y$.
\end{proof}

If $\End(G)\cong L$, then Lemma~\ref{lem:nonempty_connected}, implies that the following is well-defined. Let $y\in\mathcal{J}(L)$ and $C\subseteq\mathcal{J}(L)$ the elements covered by $y$ in $\mathcal{J}(L)$, then its \emph{private part} $P(y)$ is the connected component of $R(y)\setminus \bigcup_{x\in C}R(x)$ containing $R( y)\setminus R(\bigvee C)$. Note that $\bigcup_{x\in C}R(x)\subseteq R(\bigvee C)$.


For the next result we restrict to lattices such that their join-irreducible poset forms a lattice as well. 
\begin{lema}\label{lem:privatedisjoint}
Let $L$ be a lattice such that $\mathcal J(L)$ is a lattice and $G$ a graph such that $\End(G)\cong L$. 
If $y,z\in\mathcal{J}(L)$ are different, then $P(y)$ and $P(z)$ are disjoint.
\end{lema}
\begin{proof}
If $y<z$ and $C\subseteq \mathcal{J}(L)$ is the set of elements covered by $z$ in $\mathcal{J}(L)$, then $P(z)\subseteq R(z)\setminus \bigcup_{x\in C}R(x)$ and since $y\leq x$ for some $x\in C$ we have $P(y)\subseteq R(x)$. Hence $P(y)$ and $P(z)$ are disjoint.



If $y\parallel z$ are incomparable, then by Lemma~\ref{lem:End_is_lattice} $R(y)\cap R(z)=R(y\wedge_L z)$. Since $\mathcal J(L)$ is a lattice, we have $y\wedge_L z=y\wedge_{\mathcal J(L)} z\in\mathcal J(L)$. Indeed, $y\wedge_L z=\bigvee \{x\in\mathcal J(L)\mid x\leq y\wedge_L z\}=\bigvee \{x\in\mathcal J(L)\mid x\leq y,z\}=\bigvee \{x\in\mathcal J(L)\mid x\leq y\wedge_{\mathcal J(L)} z\}=y\wedge_{\mathcal J(L)} z\in\mathcal J(L)$. Since $y\wedge z <y,z$, both $y$ and $z$ cover in $\mathcal J(L)$ some element $x$ and $x'$, respectively, such that $y\wedge z\leq x,x'$ and $R(y\wedge z)\subseteq R(x)\cap R(x')$. By definition, $P(y)\cap R(x)=\varnothing$ and $P(z)\cap R(x')=\varnothing$. Thus, $P(y)\cap P(z)=\varnothing$.
\end{proof}

For the next lemma we need more structure on $L$. A lattice is \emph{distributive} if $x\wedge(y\vee z)=(x\wedge y)\vee(x\wedge z)$ for all $x, y, z\in L$. Birkhoff's Fundamental Theorem of Finite Distributive Lattices~\cite{B37}
says that $L$ is distributive if and only if $L\cong (\mathcal{I}(P),\subseteq)$, where  $\mathcal{I}(P)$ denotes the set of ideals of a poset $P$. Moreover, the unique such $P$ up to isomorphism is $\mathcal{J}(L)$. Every element $x\in P$ is identified with $\downarrow x\in \mathcal{I}(P)$.

An \emph{interval} in $P$ is a subset of the form $[x,z]=\{y\in P\mid x\leq y \leq z\}$ for elements $x\leq z$. The \emph{length} of an interval is the length of a shortest inclusion maximal chain. We call a poset $P$ \emph{thick} if every interval of length $2$ has least $4$ elements. See Figure~\ref{fig:thick} for an illustration.

\begin{figure}[h]
\centering
\includegraphics[width=.3\textwidth]{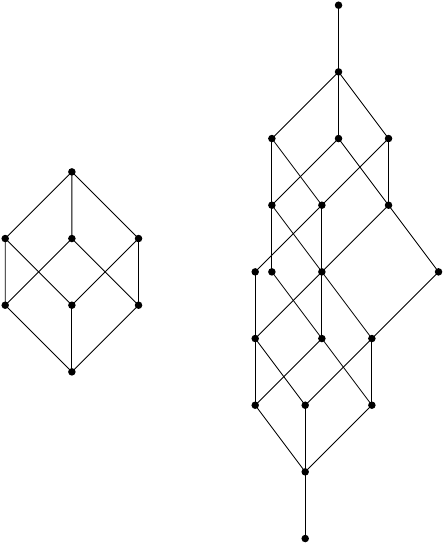}
\caption{The thick lattice $\mathcal{B}_3$ and the lattice of abstract simplicial complexes on $3$ elements $\mathcal{I}(\mathcal{B}_3)$.}\label{fig:thick}
\end{figure}

\begin{lema}\label{lem:edges}
Let $L$ be a distributive lattice such that $\mathcal{J}(L)=P$ is a thick lattice and $G$ a graph with $\End(G)\cong L$. If $x,y\in P$ such that $x\prec y$ in $P$, then there is an edge between $P(x)$ and $P(y)$.
\end{lema}
\begin{proof}
 Denote $C=\{x_1, \ldots, x_k\}\subseteq P$ the (possibly empty) set of elements covered by $y$ in $P$ except $x$ and denote by $Q(x):=R(x)\setminus R(\bigvee C)$. Note that since no element of $C$ is comparable to $x$ by distributivity we have $(\bigvee C)\wedge x=(x_1\wedge x)\vee\ldots\vee(x_k\wedge x)<x$ because every term is strictly below $x$ and $x$ is join-irreducible itself. Hence,  $Q(x)\neq\varnothing$. Moreover, note that since $P$ is  thick for all $w\prec x \prec y$ there exists a $x'\neq x$ with $w< x' \prec y$. In particular, $x'\in C$ and we get
 $\bigvee C\geq w$ for all $w\prec x$ in $P$. Thus, if $D$ are the elements covered by $x$ in $P$, then $\bigcup_{w\in D}R(x)\subseteq R(\bigvee D)\subseteq R(\bigvee C)$ and consequently $Q(x)\subseteq R(x)\setminus R(\bigvee D) \subseteq P(x)$. 
 
 We show that there is an edge from $Q(x)$ to $P(y)$. Suppose otherwise, and let $ A$ be the set of vertices of $R(y)\setminus \bigcup^k_{i=1}R(x_i)\setminus R(x)$ reachable from $Q(x)$ through vertices of $R(y)\setminus R(\bigvee C)\setminus R(x)$. This is, if $v\in A$ then there is a path starting at some vertex $u\in Q(x)$ and ending at $v$, and the rest of its vertices lie in $R(y)\setminus R(\bigvee C)\setminus R(x)$. But note that $R(y)\setminus R(\bigvee C)\setminus R(x)\subseteq R(y)\setminus \bigcup^k_{i=1}R(x_i)\setminus R(x)$. So, since $P(y)$ is a connected component of $R(y)\setminus\bigcup^k_{i=1}R(x_i)\setminus R(x)$, if $A$ and $P(y)$ intersected then $Q(x)$ and $P(y)$ would be adjacent, a contradiction. Hence, $B:=R(y)\setminus R(\bigvee C)\setminus R(x)\setminus  A$ is non-empty. Consider now the retraction $\varphi: R(y)\to R(\bigvee C)$. We can construct another retraction $\varphi':R( y)\to R(\bigvee C)\cup B$ as 
 $$v\mapsto\begin{cases}
            \varphi(v) & \text{if } v\in A\cup Q(x),\\
            v & \text{otherwise.}
\end{cases}$$
Since there are no edges from $A\cup Q(x)$ to $B$, this is a retraction, and its image $R(\bigvee C)\cup B$ is a retract. 

However, $R(\bigvee C)\subsetneq R(\bigvee C)\cup B\subsetneq R(y)$ and furthermore $R(x)\parallel R(\bigvee C)\cup B$ are incomparable. But since $y\in\mathcal J(L)$, $y$ covers a unique element in $L$. Since $C\cup x$ are exactly the elements covered by $y$ in $P$, the unique element covered by $y$ in $L$ is $\bigvee C\vee x$. This implies $R(\bigvee C)\cup B\subseteq R(\bigvee C\vee x)$.
Finally, let $z\in L$ with $\bigvee C<z\leq\bigvee C\vee x$. Since $P$ is thick, $w\leq\bigvee C$ for any $w\prec x$ in $\mathcal J(L)$, as we have seen above. Therefore, $z=\bigvee\{w\in\mathcal J(L)\mid w\leq z\}=\bigvee C\vee\bigvee\{w\in\mathcal J(L)\mid w\leq x\wedge z\}=\bigvee C\vee x$. Hence, $\bigvee C\prec\bigvee C\vee x$.



Hence, by Lemma~\ref{lem:End_is_lattice} $\bigvee C\prec\bigvee C\vee x$ and $R(\bigvee C)\subsetneq R(\bigvee C)\cup B\subseteq R(\bigvee C\vee x)$ yield
$R(\bigvee C)\cup B=R(\bigvee C\vee x)$. However, $x\leq \bigvee C\vee x$ yields $R(x)\subseteq R(\bigvee C)\cup B$ which is a contradiction to $R(x)\parallel R(\bigvee C)\cup B$.
\end{proof}

\begin{lemma}\label{lem:forcingaminor}
 Let $P$ be a thick lattice and $L=\mathcal{I}(P)$. If $G$ is a graph such that $\End(G)\cong L$, then the cover graph $G_P$ of $P$ is a minor of $G$.
\end{lemma}
\begin{proof}
 Associate every vertex $x$ of $G_P$ to $P(x)$. By Lemma~\ref{lem:nonempty_connected} each $P(x)$ is a non-empty connected subgraph of $G$. By Lemma~\ref{lem:privatedisjoint} all the $P(x)$ are mutually disjoint.  If $\{x,y\}$ is an edge in $G_P$, there is a cover relation $x\prec y$ in $P$ and by Lemma~\ref{lem:edges} there is an edge from $P(x)$ to $P(y)$. Thus, after contracting each of the connected subgraphs $P(x)$ to a single vertex and eventually deleting some edges, we obtain $G_P$ as a minor.
\end{proof}

While Lemma~\ref{lem:forcingaminor} suffices to establish one of the main results of the paper, we wonder if the assumption of $P$ being a thick lattice can be dropped. It is easy to see that lattice can be weakened to meet-semilattice without affecting our proof.

\begin{teo}\label{teo:forcingaminor}
For every positive integer $n$, there is a distributive lattice $L$ 
such that if $G$ has $\End(G)\cong L$, then $G$ contains the hypercube graph $Q_n$ as a minor.
 In particular, no class excluding a minor is endomorphism universal for commutative idempotent monoids. 
\end{teo}
\begin{proof}
 
 Consider the Boolean lattice $\mathcal B_n$, i.e., the containment order of all subsets of $\{1,...,n\}$. Every interval $[A,B]$ of $\mathcal B_n$ length $2$ has $B=A\cup\{i,j\}$, $i\neq j$, so $\mathcal B_n$ is thick. Therefore, Lemma~\ref{lem:forcingaminor} shows that any graph $G$ with $\End(G)\cong L=\mathcal I(\mathcal B_n)$ has $G_{\mathcal B_n}\cong Q_n$ as a minor. 
\end{proof}

In Theorem~\ref{teo:forcingaminor} we chose a simple family of thick lattices, namely the Boolean lattice $\mathcal{B}_n$. In this case, $\mathcal{I}(\mathcal{B}_n)$ coincides with the set of all abstract simplicial complexes on $\{1, \ldots, n\}$ ordered by inclusion. More generally, thick lattices include the class of Eulerian lattices---an important class for example containing face lattices of convex polytopes, see Stanley's survey~\cite{S94}.

Since the class of graphs of maximum degree $2$  excludes $K_{1,3}$ as a minor, a trivial consequence of Theorem~\ref{teo:forcingaminor} is that Theorem~\ref{teo:mainboundeddegree} is best-possible:
\begin{corol}\label{corol:bestpossible}
Graphs of maximum degree $2$ are not $\End$-universal for (distributive) lattices.
\end{corol}

\section{Forcing a topological minor}
A semigroup is \emph{completely regular} if it can be expressed as a union of subgroups.
In the present section we show that no class excluding a topological minor can be $\End$-universal for the class of completely regular monoids (Theorem~\ref{teo:forcingatopologicalminor}). Our proof is inspired by the proof of the main result of~\cite{BP80}, showing the analogous (weaker) result for the class of all monoids. 
So we begin introducing some terminology from \cite{BP80}. Let $M$ be a monoid, $V$ a set and $V^{V}$ the monoid of all transformations of $V$. If $\Phi:M\rightarrow V^{V}$ is an injective monoid homomorphism, then $\Phi$ is called a \emph{faithful representation of $M$}. Now let $\Omega$ be a set and assume that $M$ is a submonoid of $\Omega^{\Omega}$. A faithful representation $\Phi:M\rightarrow V^V$ is called a \emph{pseudorealization} if there is an injective mapping of $\Omega$ into $V$ ($x\mapsto \overline x$, say), such that $(\Phi\alpha)\overline x=\overline{\alpha x}$ for every $x\in\Omega$ and $\alpha\in M$. (Here we drop some parentheses as in~\cite{BP80}.) Finally, in such case we say that the transformation monoid $\Phi M$ is a \emph{pseudorealization of $M$}.

Now take $\Omega=\mathbb Z_p\times\{1,2,3\}$, where $p$ is a prime number. Let $\Phi:\mathbb Z_p^2\rightarrow\Omega^{\Omega}$ be the faithful representation defined by $(c,d)\mapsto\pi_{c,d}$, where
\begin{align*}
    \pi_{c,d}(a,1) &= (a+c,1) \\
    \pi_{c,d}(a,2) &= (a+d,2) \\
    \pi_{c,d}(a,3) &= (a+c+d,3)
\end{align*}
for every $a\in\mathbb Z_p$. Consider the permutation group $P=\Phi\mathbb Z^2_p$. Then we get~\cite[Lemma 3.2]{BP80}:

\begin{lemma}\label{prop:BP80} Let $p$, $\Omega$ and $P$ as above. Assume that $G$ is a graph and $\Aut(G)$ is a pseudorealization of $P$. Then $G$ contains a subdivision of $K_{p,p}$.
\end{lemma}

Next, we strengthen \cite[Lemma 3.1]{BP80}. In fact, most of their arguments can be reused.

\begin{lemma}\label{lem:3.1} Let $p$, $\Omega$ and $P$ as above. There exists a completely regular monoid $M\subseteq\Omega^{\Omega}$  
such that 
\begin{enumerate}[(i)]
   \item $P$ coincides with the group of invertible elements of $M$;
   \item every faithful representation of $M$ is a pseudorealization.
\end{enumerate}
\end{lemma}
\begin{proof} Let $\textrm{S}_p$ be the symmetric group over $\{1,...,p\}$. Consider the subsets $C=\{\gamma_x\in\Omega^{\Omega}\mid x\in\Omega \text{ and } \gamma_x(\Omega)=\{x\}\}$ and $S=\{\alpha\in\Omega^{\Omega}\mid \exists\sigma_{\alpha,1},\sigma_{\alpha,2},\sigma_{\alpha,3}\in\mathrm S_p \text{ and } \forall k\in\mathbb Z_p\ \alpha(k,i)=(\sigma_{\alpha,i}(k),1)\}$. Set $M=S\cup P\cup C$.  
We claim that $M$ is a completely regular monoid. Its neutral element is the neutral element of $P$, which clearly satisfies $P^2\subseteq P$ and is a group, hence completely regular. Moreover, $C^2,CP,CS,PC,SC\subseteq C$ and $C$ is idempotent, hence completely regular. Let $k\in\mathbb Z
_p$ and $i\in\{1,2,3\}$. If $\pi_{c,d}\in P$ and $\alpha\in S$, then $(\pi_{c,d}\circ\alpha)(k,i)=(\sigma_{\alpha,i}(k)+c,1)$ and $(\alpha\circ\pi_{c,d})(k,i)=(\sigma_{\alpha,i}(k+e),1)$, where $$e=\begin{cases}c & \text{if } i=1  \\ d & \text{if } i=2 \\ c+d &\text{if } i=3.\end{cases}$$ Therefore $PS,SP\subseteq S$. Moreover, if $\alpha,\beta\in S$ then $(\alpha\circ\beta)(k,i)=(\sigma_{\alpha,1}(\sigma_{\beta,i}(k)),1)$, so $S^2\subseteq S$. In particular, we get that if $\sigma_{\alpha,1}\in\textrm S_p$ is of order $m$, then $\alpha^{m+1}=\alpha$, hence $S$ is completely regular. We conclude that $M$ is a completely regular monoid. It is clear that (i) holds.

Let $V$ be a set and $\Psi:M\rightarrow V^V$ a faithful representation of $M$. We are now going to prove that $\Psi$ is a pseudorealization. Take two different elements $(a,1),(b,1)\in\mathbb Z_p\times\{1\}\subseteq\Omega$. Since $\Psi\gamma_{(a,1)}\neq\Psi\gamma_{(b,1)}$, we have that $(\Psi\gamma_{(a,1)})v\neq(\Psi\gamma_{(b,1)})v$ for some $v\in V$. Define the mapping
\begin{align*}
\Omega &\ \rightarrow\ V \\ 
x &\ \mapsto \ \overline x=(\Psi\gamma_x)v.
\end{align*}
This mapping is injective. Indeed, let $(k,i),(\ell,j)$ be two distinct elements of $\Omega=\mathbb Z_p\times\{1,2,3\}$. If $k\neq\ell$ then take a permutation $\sigma\in\textrm{S}_p$ with $\sigma(k)=a$, $\sigma(\ell)=b$, and the transformation
$\alpha\in S$ with $\sigma_{\alpha,1}=\sigma_{\alpha,2}=\sigma_{\alpha,3}=\sigma$. Then 
\[(\Psi\alpha)\overline{(k,i)}=(\Psi\alpha)(\Psi\gamma_{(k,i)})v=(\Psi(\alpha\gamma_{(k,i)}))v=(\Psi\gamma_{(a,1)})v\neq\]
\[(\Psi\gamma_{(b,1)})v=(\Psi(\alpha\gamma_{(\ell,j)}))v=(\Psi\alpha)(\Psi\gamma_{(\ell,j)})v=(\Psi\alpha)\overline{(\ell,j)},\]
so $\overline{(k,i)}\neq\overline{(\ell,j)}$. And if $i\neq j$ then take $\sigma,\tau\in\textrm{S}_p$ with $\sigma(k)=a$, $\tau(\ell)=b$, and a transformation $\alpha\in S$ with $\sigma_{\alpha,i}=\sigma$ and $\sigma_{\alpha,j}=\tau$. Repeating the above argument we see that again $\overline{(k,i)}\neq\overline{(\ell,j)}$. Moreover, for any $\beta\in M$ and $x\in\Omega$
\[(\Psi\beta)\overline x=(\Psi\beta)(\Psi\gamma_x)v=(\Psi(\beta\gamma_x))v=(\Psi\gamma_{\beta x})v=\overline{\beta x}.\]
Hence $\Psi$ is a pseudorealization.
\end{proof}

Note that if $M$ is a monoid, $P\subseteq M$ is its group of invertible elements, $G$ is a graph and $\Psi:M\rightarrow\End(G)$ is a monoid isomorphism then $\Psi_{|P}:P\rightarrow\Aut(G)$ is a faithful representation. Therefore, from the above lemmas we obtain: 

\begin{teo}\label{teo:forcingatopologicalminor}
For every positive integer $n$, there is a completely regular monoid $M$ 
 such that if $G$ has $\End(G)\cong M$, then $G$ contains a subdivision of $K_{n,n}$. 
 In particular, no class excluding a topological minor is endomorphism universal for completely regular monoids. 
\end{teo}



\section{Comments}

The results of this paper extend the study of endomorphism universality of sparse graph classes. Particular attention is paid to classes of monoids that are important in Semigroup Theory. Let us first point out some natural follow-up questions to our results and then present some other problems of a similar flavor. 

We have shown that commutative idempotent monoids (Theorem~\ref{teo:mainboundeddegree}) as well as $k$-cancellative monoids (Theorem~\ref{teo:groups}) can be represented as endomorphism monoids of graphs of bounded degree, while no class excluding a topological minor can represent all completely regular monoids (Theorem~\ref{teo:forcingatopologicalminor}). It is thus a natural question what happens with other classes of monoids, in particular completely regular classes that contain groups or commutative idempotent monoids. Indeed, a second part of the question of Babai and Pultr~\cite[Problem 1.4]{BP80} asks whether every idempotent monoid is the endomorphism monoid of a graph of bounded degree. The structure of idempotent monoids is understood in following sense: every idempotent monoid can be seen as a lattice $L$ where element $x\in L$ corresponds to a \emph{rectangular band} $S_x$, i.e., a direct product of a left-zero semigroup $L_{x_{\ell}}$ and a right-zero semigroup $R_{x_r}$. Recall that $L_{x_{\ell}}$ is defined on the set $\{\ell_1, \ldots, \ell_{x_{\ell}}\}$ by $\ell_i\cdot \ell_j=\ell_i$ and similarly $R_{x_r}$ is defined on the set $\{r_1, \ldots, r_{x_r}\}$ by $r_i\cdot r_j=r_j$. Now, if $m\in  S_x=L_{x_{\ell}}\times R_{x_r}$ and $m' \in S_{x'}= L_{x'_{\ell}}\times R_{x'_r}$, then $mm'\in S_{x\wedge x'}$, see~\cite{M54}. Theorem~\ref{teo:mainboundeddegree} can be seen as the extremal case where $x_{\ell}=x_r=1$ for all $x\in L$. We sketch the other extremal case where $|L|=2$. In this case $M$ is obtained by adding a neutral element to a rectangular band $S=L_{x_{\ell}}\times R_{x_r}$, i.e., $M$ is the \emph{adjoint monoid} $S^+$ of $S$. Figure~\ref{fig:R5L3} sketches how to construct an arc-colored graph $G'$ with $\End(G')\cong S^+$ if $S$ is a rectangular band. 

\begin{figure}[h]
\centering
\includegraphics[width=\textwidth]{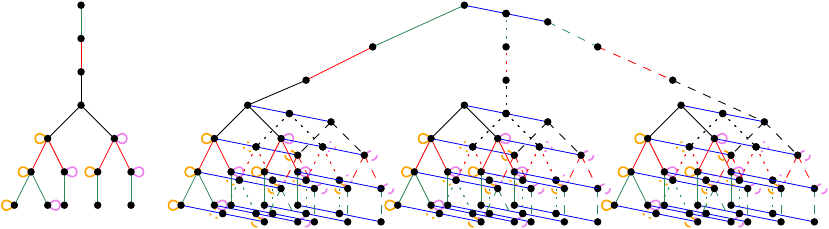}
\caption{Left: an arc-colored graph $G$ with $\End(G)\cong R_5^+$. Right: an arc-colored graph $G'$ with  $\End(G')\cong(L_3\times R_5)^+$.}\label{fig:R5L3}
\end{figure}

Thus, our results from Sections~\ref{sec:HeldrlinPultr} and~\ref{sec:blowup} yield:
\begin{prop}\label{prop:band}
The adjoint monoid of a rectangular band is the endomorphism monoid of a graph of bounded degree.
\end{prop}
We do not know how to use this with Theorem~\ref{teo:mainboundeddegree} to represent all idempotent monoids by graphs of bounded degree, so~\cite[Problem 1.4]{BP80} remains open. Similar questions that arise from Figure~\ref{fig:venn} are whether the class of bounded degree graphs is $\End$-universal for monoids that are commutative and regular (i.e.~for every $x$ there is a $y$ such that $xyx = x$) or monoids that are inverse (i.e.~for every $x$ there is a unique $y$ such that $x=xyx$ and $y=yxy$).
Koubek, R{\"o}dl, and Shemmer~\cite{KRS09}, when studying the minimum number of vertices or edges required for a graph to represent a given monoid, obtain special results for the class of completely simple monoids---a subclass of completely regular monoids. So this would be another natural candidate.

\bigskip


Conversely to the above problems, one may wonder what are the graphs whose endomorphism monoid falls into a given class $\mathcal{M}$. We have presented an easy result (Lemma~\ref{lem:End_is_lattice}) that characterizes graphs whose endomorphism monoid is idempotent and commutative in an algebraic manner. In the same vein one may characterize graphs whose endomorphism monoid is idempotent, as those graphs all of whose endomorphisms are retractions. On the other hand $\End(G)$ is a group if and only if $G$ has no non-trivial retracts, see e.g.~\cite{Kna-19,HN04}.  However, purely graph theoretic characterizations seem difficult in the above cases. A question of Knauer and Wilkeit~\cite[Question 4.2]{M88} asks for a characterization of graphs with regular endomorphism monoid. This remains open, see~\cite{HL08}. Similarly it seems interesting to characterize graphs whose endomorphism monoid is completely regular, completely simple, commutative, or inverse.

A fundamental tool for representing a given monoid $M$ in a desired way is finding a generating set $C\subseteq M$ such that $\Cay_{\mathrm{col}}(M,C)$ is well-behaved (Lemma~\ref{lem:basic}). It is thus a natural question which monoids or semigroups (of a particular class) $\mathcal{M}$ admit a generating set such that their Cayley graph, i.e., a \emph{generated} Cayley graph, falls into a particular class of graphs $\mathcal{G}$. When sparse classes are considered, usually one only considers the underlying simple Cayley graph $\underline{\Cay}(M,C)$, where colors, orientations, multiplicity, and loops are ignored.
An old result into this direction by Maschke~\cite{M96} characterizes the planar groups, i.e., those groups that have a planar generated Cayley graph. The planar \emph{right-groups}, i.e., a product of a group and a right-zero semigroup, were characterized in~\cite{KK16}, a characterization of certain planar Clifford semigroups was given by Zhang~\cite{Zha-08} and slightly corrected in~\cite[Theorem 13.3.5]{Kna-19}, and Solomatin characterized planar products of cyclic semigroups~\cite{Sol-06} and studied some classes of outerplanar semigroups~\cite{Sol-11}. It remains open to characterize outerplanar semigroups, see~\cite[Problem 2]{KK16}.


Another current is the question which graphs in a given class $\mathcal{G}$ are (generated) Cayley graphs of a semigroup of prescribed type. The question which Generalized Petersen Graphs are underlying simple graphs of Cayley graphs of monoids has been studied in~\cite{GMK21}. A characterization of planar (minimally) generated Cayley graphs of commutative idempotent semigroups (a.k.a semilattices) has been given in~\cite[Theorem 13.3.4]{Kna-19}. Results of Zelinka~\cite{Zelinka1981} that were corrected in~\cite{KP21} imply that all forests are underlying simple Cayley graphs of a monoid, while not every tree is a generated Cayley graph of a monoid~\cite{KP21}. It is open, whether every graph of treewidth $2$ (a.k.a series-parallel graph) is the underlying simple Cayley graph of a monoid, see~\cite[Question 6.2]{KP21}. While there are graphs of treewidth $3$ that are not the underlying simple Cayley graph of monoid, it remains open, whether every graph is the underlying simple Cayley graph of a semigroup, see~\cite[Question 6.3]{KP21}.


\paragraph{Acknowledgments:} We thank Pavol Hell for pointing us to the generalization of Frucht's theorem to endomorphisms~\cite{HN73} and Jarik Ne\v{s}et\v{r}il and Patrice Ossona de Mendez for informing us about their work on $\End$-universal classes of bounded expansion~\cite{NOdM17}. KK wishes to thank Ignacio Garc\'ia-Marco and Irene M\'arquez-Corbella for some initial discussions on the subject of this paper. KK was partially supported by the French \emph{Agence nationale de la recherche} through
project ANR-17-CE40-0015 and by the Spanish \emph{Ministerio de Econom\'ia, Industria y Competitividad}
through grant RYC-2017-22701, the Severo Ochoa and Mar\'ia de Maeztu Program for Centers and Units of Excellence in R\&D (CEX2020-001084-M), and grant PID2022-137283NB-C22. GPS was also partially supported by the latter grant, and by project ANR-21-CE48-0012.
\small
\bibliography{biblio.bib}
\bibliographystyle{my-siam}

\end{document}